\documentclass[letterpaper, 11pt,  reqno]{amsart}

\usepackage[margin=1.1in,marginparwidth=1.5cm, marginparsep=0.5cm]{geometry}

\usepackage{amsmath,amssymb,amscd,amsthm,amsxtra, esint}

\usepackage{kotex}
\usepackage{mathrsfs} 

\usepackage[implicit=true]{hyperref}

\usepackage{color} 
\usepackage{ulem}
\usepackage[makeroom]{cancel}

\allowdisplaybreaks[2]

\definecolor{gr}{rgb}   {0.,   0.69,   0.23 }
\definecolor{bl}{rgb}   {0.,   0.5,   1. }
\definecolor{mg}{rgb}   {0.85,  0.,    0.85}
\definecolor{yl}{rgb}   {0.8,  0.7,   0.}
\definecolor{or}{rgb}  {0.7,0.2,0.2}

\newtheorem{theorem}{Theorem} [section]

\newtheorem{lemma}[theorem]{Lemma}
\newtheorem{proposition}[theorem]{Proposition}
\newtheorem{remark}[theorem]{Remark}

\newtheorem{definition}[theorem]{Definition}

\DeclareMathOperator*{\supp}{supp}

\DeclareMathOperator{\Id}{Id}

\newcommand{\I}{\hspace{0.5mm}\text{I}\hspace{0.5mm}}
\newcommand{\II}{\text{I \hspace{-2.8mm} I} }

\newcommand{\noi}{\noindent}
\newcommand{\Z}{\mathbb{Z}}
\newcommand{\R}{\mathbb{R}}

\newcommand{\T}{\mathbb{T}}

\let\P= \undefined
\newcommand{\P}{\mathbf{P}}

\newcommand{\E}{\mathbb{E}}

\newcommand{\F}{\mathcal{F}}

\newcommand{\al}{\alpha}
\newcommand{\be}{\beta}
\newcommand{\dl}{\delta}

\newcommand{\nb}{\nabla}

\newcommand{\Dl}{\Delta}
\newcommand{\eps}{\varepsilon}

\newcommand{\g}{\gamma}
\newcommand{\G}{\Gamma}
\newcommand{\ld}{\lambda}

\newcommand{\s}{\sigma}

\newcommand{\ft}{\widehat}

\newcommand{\wt}{\widetilde}
\newcommand{\cj}{\overline}

\newcommand{\dt}{\partial_t}

\newcommand{\too}{\longrightarrow}

\renewcommand{\l}{\ell}
\renewcommand{\o}{\omega}
\renewcommand{\O}{\Omega}

\newcommand{\les}{\lesssim}

\newcommand{\jb}[1]
{\langle #1 \rangle}

\newcommand{\ind}{\mathbf 1}

\newcommand{\M}{\mathcal{M}}

\newcommand{\N}{\mathbb{N}}

\newtheorem*{ackno}{Acknowledgements}

\numberwithin{equation}{section}
\numberwithin{theorem}{section}

\newcommand{\PP}{\mathbb{P}}

\DeclareMathOperator{\Law}{Law}

\newcommand{\W}{\mathcal{W}}

\newcommand{\dr}{\theta}
\newcommand{\Dr}{\Theta}

\newcommand{\Ha}{\mathbb{H}_a}

\usepackage{tikz}

\usetikzlibrary{shapes.misc}
\usetikzlibrary{shapes.symbols}
\usetikzlibrary{shapes.geometric}
\tikzset{
	dot/.style={circle,fill=black,draw=black,inner sep=1pt,minimum size=0.5mm},
	>=stealth,
	}
\tikzset{
	ddot/.style={circle,fill=white,draw=black,inner sep=2pt,minimum size=0.8mm},
	>=stealth,
	}

\tikzset{decision/.style={ 
        draw,
        diamond,
        aspect=1.5
    }}

\tikzset{dia2/.style
={diamond,fill=white,draw=black,inner sep=0pt,minimum size=1mm},
	>=stealth,
	}

\tikzset{dia/.style
={star,fill=black,draw=black,inner sep=0pt,minimum size=1mm},
	>=stealth,
	}

\colorlet{symbols}{black}
\colorlet{testcolor}{green!60!black}

\def\1{\mathbf{{1}}}

\usetikzlibrary{shapes.misc}
\usetikzlibrary{shapes.symbols}
\usetikzlibrary{snakes}
\usetikzlibrary{decorations}
\usetikzlibrary{decorations.markings}
\definecolor{dblue}{rgb}{0.1, 0.1, 0.9}

\tikzset{
	root/.style={circle,fill=testcolor,inner sep=0pt, minimum size=2mm},		
	dot/.style={circle,fill=black,draw=black, solid,inner sep=0pt,minimum size=0.75mm},
	bdot/.style={circle,fill=blue,draw=dblue, solid,inner sep=0pt,minimum size=0.75mm},
		}
\colorlet{symbols}{blue!90!black}

\makeatletter
\def\DeclareSymbol#1#2#3{\expandafter\gdef\csname MH@symb@#1\endcsname{\tikz[baseline=#2,scale=0.15]{#3}}%
\expandafter\gdef\csname MH@symb@#1s\endcsname{\scalebox{0.6}{\tikz[baseline=#2,scale=0.15]{#3}}}}
\def\<#1>{\csname MH@symb@#1\endcsname}
\makeatother

\DeclareSymbol{sunset}{-3}{\draw (0,0) circle [x radius = 1.5, y radius = 1]; \draw (-1.5,0) node[dot] {} -- (1.5,0) node[dot] {};}
\DeclareSymbol{tadpole}{-3}{\draw (0,0) circle [x radius = 0.75, y radius = 1]; \draw (0,-1) node[dot] {};}

\DeclareSymbol{1}{0}{\draw[white] (-.4,0) -- (.4,0); \draw (0,0)  -- (0,1.2) node[dot] {};}
\DeclareSymbol{2}{0}{\draw (-0.5,1.2) node[dot] {} -- (0,0) -- (0.5,1.2) node[dot] {};}
\DeclareSymbol{3}{0}{\draw (0,0) -- (0,1.2) node[dot] {}; \draw (-.7,1) node[dot] {} -- (0,0) -- (.7,1) node[dot] {};}
\DeclareSymbol{4}{0}{\draw (-0.36,1.2) node[dot] {} -- (0,0) -- (0.36,1.2) node[dot] {}; \draw (-1,1) node[dot] {} -- (0,0) -- (1,1) node[dot] {};}
\DeclareSymbol{30}{-3}{\draw (0,0) -- (0,-1); \draw (0,0) -- (0,1.2) node[dot] {}; \draw (-.7,1) node[dot] {} -- (0,0) -- (.7,1) node[dot] {};}
\DeclareSymbol{31}{-3}{\draw (0,0) -- (0,-1) -- (1,0) node[dot] {}; \draw (0,0) -- (0,1.2) node[dot] {}; \draw (-.7,1) node[dot] {} -- (0,0) -- (.7,1) node[dot] {};}
\DeclareSymbol{32}{-3}{\draw (0,0) -- (0,-1) -- (1,0) node[dot] {}; \draw (0,0) -- (0,-1) -- (-1,0) node[dot] {}; \draw (0,0) -- (0,1.2) node[dot] {}; \draw (-.7,1) node[dot] {} -- (0,0) -- (.7,1) node[dot] {};}
\DeclareSymbol{20}{-3}{\draw (0,0) -- (0,-1);\draw (-.7,1) node[dot] {} -- (0,0) -- (.7,1) node[dot] {};}
\DeclareSymbol{22}{-3}{\draw (0,0.3) -- (0,-1) -- (1,0) node[dot] {}; \draw (0,0.3) -- (0,-1) -- (-1,0) node[dot] {};\draw (-.7,1) node[dot] {} -- (0,0.3) -- (.7,1) node[dot] {};}
\DeclareSymbol{202}{-3}{\draw (-.6625,0) -- (0,-1) -- (.6625,0)  {}; \draw (-1.1,1) node[dot] {} -- (-.6625,0) -- (-0.365,1) node[dot] {}; \draw (0.365,1) node[dot] {} -- (.6625,0) -- (1.1,1) node[dot] {};}

\makeatletter
\def\DeclareSymbol#1#2#3{\expandafter\gdef\csname MH@symb@#1\endcsname{\tikz[baseline=#2,scale=0.15]{#3}}}
\def\<#1>{\csname MH@symb@#1\endcsname}
\makeatother

\DeclareSymbol{X}{-2.4}{\node[dot] {};}
\DeclareSymbol{1}{0}{\draw[white] (-.4,0) -- (.4,0); \draw (0,0)  -- (0,1.2) node[dot] {};}
\DeclareSymbol{2}{0}{\draw (-0.5,1.2) node[dot] {} -- (0,0) -- (0.5,1.2) node[dot] {};}

\DeclareSymbol{sunset}{-3}{\draw (0,0) circle [x radius = 1.5, y radius = 1]; \draw (-1.5,0) node[dot] {} -- (1.5,0) node[dot] {};}
\DeclareSymbol{tadpole}{-3}{\draw (0,0) circle [x radius = 0.75, y radius = 1]; \draw (0,-1) node[dot] {};}

\DeclareSymbol{T0}{-2.7}
 { \draw (0,0) node[ddot]{};}

\DeclareSymbol{T1}{0}
 {  
 \draw (0,0) node[dot]{} -- (0,4) node[ddot] {}; 
 \draw (-3,0) node[dot] {} -- (0,4)node[ddot] {} -- (3,0) node[dot] {};
\draw[line width=1pt] (0, 4)node[ddot, label=above:$r_1$]{} 
.. controls(3,6) and (5,6) ..
(8, 4)node[ddot, label=above:$r_2$]{}
(4, 12) node[label ] {$\underline{j = 1}$};
 }

\DeclareSymbol{T2}{0}
 {\draw (0,0) node[dot]{} -- (0,4) node[ddot] {}; 
 \draw (-3,0) node[dot] {} -- (0,4)node[ddot] {} -- (3,0) node[dot] {};
 \draw (0,-4) node[dot]{} -- (0,0) node[dot] {}; 
 \draw (-3,-4) node[dot] {} -- (0,0)node[dot] {} -- (3,-4) node[dot] {};
\draw[line width=1pt] (0, 4)node[ddot, label=above:$r_1$]{} 
.. controls(3,6) and (5,6) ..
(8, 4)node[ddot, label=above:$r_2$]{}
(4, 12) node[label ] {$\underline{j = 2}$};
 }

\DeclareSymbol{T3}{0}
 {\draw (0,0) node[dot]{} -- (0,4) node[ddot] {}; 
 \draw (-3,0) node[dot] {} -- (0,4)node[ddot] {} -- (3,0) node[dot] {};
 \draw (0,-4) node[dot]{} -- (0,0) node[dot] {}; 
 \draw (-3,-4) node[dot] {} -- (0,0)node[dot] {} -- (3,-4) node[dot] {};
\draw (10,0) node[dot]{} -- (10,4) node[ddot] {}; 
 \draw (7,0) node[dot] {} -- (10,4)node[ddot] {} -- (13,0) node[dot] {};
\draw[line width=1pt] (0, 4)node[ddot, label=above:$r_1$]{} 
.. controls(4,6) and (6,6) ..
(10, 4)node[ddot, label=above:$r_2$]{}
(5, 12) node[label ] {$\underline{j = 3}$};
 }

\tikzstyle{dot1} = [ draw=  gray!00, 
 rectangle, rounded corners, fill=gray!00,  inner sep=1pt, inner ysep=3pt]

\tikzstyle{dot2} = [ draw=  black, 
ellipse, fill=gray!00,  inner sep=1pt, inner ysep=3pt]

\tikzstyle{dot3} = [ draw=  gray!00, 
ellipse, fill=gray!00,  inner sep=1pt, inner ysep=3pt]

\DeclareSymbol{T4}{0}
 {\draw (0,0) node[dot1]{$\phantom{n_2^{(1)} = n^{(2)}}$} -- (0,15) node[dot1] {$\phantom{n^{(1)}}$}; 
 \draw (-13,0) node[dot1] {$n_1^{(1)}$} -- (0,15)node[ddot] {$\phantom{n^{(1)}}$} 
 -- (13,0) node[dot1] {$n_3^{(1)}$};
 \draw (0,-15) node[dot1]{$n_2^{(2)}$} -- (0,0) node[dot1] {$\phantom{n_2^{(1)} = n^{(2)}}$}; 
 \draw (-13,-15) node[dot1] {$n_1^{(2)}$} -- (0,0)node[dot1] {$n_2^{(1)} = n^{(2)}$} -- (13,-15) node[dot1] {$n_3^{(2)}$};
\draw (40,0) node[dot1]{{$n_2^{(3)}$}} -- (40,15) node[dot3]  {$\phantom{n^{(1)} = n^{(3)}}$} ; 
 \draw (27,0) node[dot1] {$n_1^{(3)}$} -- (40,15)node[dot3]  {$\phantom{n^{(1)} = n^{(3)}}$}  -- (53,0) node[dot1] {$n_3^{(3)}$};
\draw[line width=1pt] (0, 15)node[ddot, label=above:$r_1$]{$n^{(1)}$} 
.. controls(10,23) and (27,23) ..
(40, 15)node[dot2, label=above:$r_2$] {$n^{(1)} = n^{(3)}$} ;
 }

\makeatletter
\def\DeclareSymbol#1#2#3{\expandafter\gdef\csname MH@symb@#1\endcsname{\tikz[baseline=#2,scale=0.15]{#3}}}
\def\<#1>{\csname MH@symb@#1\endcsname}
\makeatother

\begin{document}
\baselineskip = 14pt

\title[Large deviations and free energy of the $\Phi^3$ measure ]
{Large deviations and free energy of Gibbs measure for the dynamical 
$\Phi^3$-model  in Infinite Volume}


\author[K.~Seong and P.Sosoe]
{Kihoon Seong and Philippe Sosoe}

\address{Kihoon Seong\\
Department of Mathematics\\
Cornell University\\ 
310 Malott Hall\\ 
Ithaca\\ New York 14853\\ 
USA }

\email{kihoonseong@cornell.edu}

\address{Philippe Sosoe\\
Department of Mathematics\\
Cornell University\\ 
310 Malott Hall\\ 
Cornell University\\
Ithaca\\ New York 14853\\ 
USA }

\email{ps934@cornell.edu}



\begin{abstract}

We study the large deviations for focusing Gibbs measures by analyzing the asymptotic behavior of the free energy in the infinite volume limit.  This is the invariant Gibbs measure for the dynamical $\Phi^3_2$-models. From our sharp estimates for the partition function, we establish a concentration phenomenon of the $\Phi^3_2$-measure around the zero field, leading to a triviality result in the infinite volume: the ensemble collapses onto a delta function on the zero field.







\end{abstract}



\maketitle


\tableofcontents

\setlength{\parindent}{0mm}
\setlength{\parskip}{6pt}

\section{Introduction}

\subsection{Motivation for the main result}

We study a measure on two-dimensional distributions inspired by Euclidean quantum field theory, the so-called $\Phi^3_2$ scalar field theory. This measure is defined on the space $\mathcal{D}'(\T^2_L)$  of Schwartz  distributions, and is formally
written as 
\begin{align}
d\rho_L(\phi)=Z_{L}^{-1}e^{-\frac{\s}{3}\int_{\T^2_{L}}\phi(x)^3 dx-\frac 12\int_{\T^2_L}|\nb \phi(x)|^2 dx} \prod_{x\in \T_L^2} d\phi(x).
\label{Gibbs1}
\end{align}

\noi
Here $Z_L$ is the partition function, $\T^2_L=(\R / L\Z)^2$ is a dilated torus of sidelength $L\in \N$, $\prod_{x\in\T^2_L}d\phi(x)$ is the (non-existent) Lebesgue measure on fields $\phi: \T^2_L \to \R$, and $\s \in \R\setminus\{0\}$ is the coupling constant measuring the strength of the cubic interaction potential\footnote{ Compared to the $\Phi^4$ theory, the cubic interaction $\phi^3$ is not sign-definite and so, the sign of the coupling constant $\s$ plays no significant role. Therefore, we assume $\s \in \mathbb{R} \setminus\{0\}$.}.
Despite their apparent instability in the large field regime, in the physics literature, field theories with cubic interaction commonly appear as toy models \cite[Chapter 9]{srednicki}. Tensor fields with cubic interaction have also appeared in renormalization group analyses of the Potts mode \cite{PL,WJ}. The main result of this paper is a concentration estimate for the measures $\rho_L$  in \eqref{Gibbs1} in the infinite volume limit $L\to \infty$, from which we deduce the triviality of the $\Phi^3_2$-measure in that limit.


When the cubic interaction is replaced by a higher-order term $\s\phi^k$, where $k \ge 4$ is odd with $\s\in \R \setminus\{0\}$, or $k \ge 4$ is even with $\s<0$ (the so-called focusing interaction), the corresponding measure cannot be constructed, even with proper microcanonical or grand canonical considerations, as shown by Brydges and Slade \cite{BS}; see also \cite{OSeoT}.  This failure of the measure construction for higher-order focusing interactions isolates the cubic case as the remaining model where a meaningful rigorous formulation remains possible.

From the perspective of Euclidean quantum field theory, the $\Phi^3_2$ measure on the finite volume  $\T^2=(\R / \Z)^2$ was studied by Jaffe \cite{Jaffe} in the form
\begin{align}
 e^{-\frac{\s}{3} \int_{\T^2} :\, \phi^3: \, dx } \ind_{\{\int_{\T^2} :\,\phi^2: \, dx\, \leq K\}}  d \mu(\phi),
\label{Jaffe1}
\end{align}

\noi
where $K>0$ and $\mu$ is the free field with covariance $(-\Dl)^{-1}$.
Here, $:\phi^3:$ and $:\phi^2:$ denote Wick renormalizations which are necessary due to the singular nature of the free field. See also Brydges and Slade \cite{BS} for an explanation of the $\Phi^3_2$ measure.
As studied in the previous works of Lebowitz, Rose, and Speer \cite{LRS}, such focusing Gibbs ensembles are necessarily microcanonical in the particle number $\int |\phi|^2 dx$, since the canonical Gibbs ensemble $e^{-H(\phi)}\prod_{x}d\phi(x)$ with respect to the Hamiltonian 
\begin{align}
H(\phi)=\frac 12 \int_{\T_L^2 } |\nb \phi|^2 dx+\frac{\s}{3}\int_{\T_L^2}\phi^3(x) dx
\label{Ham0}
\end{align}

\noi 
cannot be constructed for any $\s \in \R \setminus\{0\}$ without the conditioning given in \eqref{Jaffe1}. This is because the Hamiltonian $H(\phi)$ in \eqref{Ham0} is not bounded from below, due to its focusing nature $\phi^3$,  which implies $H(\phi)$  tends to $-\infty$  along certain directions in the phase space.




Although the $\Phi^3$ measure \eqref{Jaffe1} is of interest in quantum field theory as studied by Jaffe, it does not arise as an invariant measure for any dynamics with a Gibbsian structure. See Remark \ref{REM:not} for further explanation.  In \cite{BO95, CFL}, Bourgain and  Carlen-Fr\"ohlich-Lebowitz  instead proposed to consider the grand-canonical Gibbs measure of the form
\begin{align}
d\rho(u) = Z^{-1}
e^{-\frac{\s}{3} \int_{\T^2} :\, \phi^3 :\, dx  -A\big( \int_{\T^2} :\, \phi^2: \, dx   \big)^2  }  d \mu(u)
\label{GGibbs}
\end{align}

\noi
for sufficiently large $A>0$. Here, the parameter $A>0$ is sometimes known as a chemical potential, by analogy with statistical mechanics. The grand-canonical Gibbs measure  \eqref{GGibbs} can be interpreted as the equilibrium state of the parabolic/hyperbolic $\Phi^3_2$-model; see Remark \ref{REM:not}.
Moreover, the choice of the exponent $\g=2$ in the taming term  $A(\cdot)^\g$  in \eqref{GGibbs} is optimal; see Remark \ref{REM:GNS}.


Our paper is a continuation of the study of the grand-canonical $\Phi^3_2$-measure on $\T^2_L$
\begin{align}
d\rho_L(u) = Z_{L}^{-1}
e^{-\frac{\s}{3} \int_{\T^2_L} :\, \phi^3 :\, dx  -A\big( \int_{\T^2_L} :\, \phi^2: \, dx   \big)^2  }  d \mu_L(u),
\label{GGibs1}
\end{align}

\noi
examining in particular the behavior of the measure in the infinite volume limit as $L \to \infty$, where  $\mu_L$ is the free field on $\T^2_L$.
The study of the infinite volume limit of the one-dimensional focusing Gibbs measure was first initiated by McKean \cite{McKean} and later developed further by Rider \cite{Rider}. Their work  examines the concentration behavior of the focusing Gibbs measure on $\T_L$ as $L\to \infty$, showing the collapse of the measure onto the zero field. In this paper, we explore the corresponding phenomenon in two dimensions, for the only measure with a polynomial interaction for which the problem appears well-posed.  We first state our main result in a somewhat informal manner. See Theorem \ref{THM:0} for the precise statement.


\begin{theorem}\label{THM:00}
Let $\rho_L$ be the grand-canonical $\Phi^3_2$ measure \eqref{GGibs1} on finite volume $\T^2_L$. Given any $\s\in \R \setminus\{0\}$,  there exists a constant $A_0=A_0(\s) \ge 1$ independent of $L \ge 1$ such that for  all $A \ge A_0$,
\begin{itemize}
\item[(i)] we have the non-volume order large deviation
\begin{align}
\lim_{L\to \infty} \frac{\log Z_{L}}{L^4}=-\inf_{\phi\in H^1 (\R^2)} H(\phi),
\label{free0}
\end{align}

\noi
where 
\begin{align}
H(\phi)&=\frac 12\int_{\R^2} |\nb \phi|^2 dx+\frac{\s}{3} \int_{\R^2} \phi^3 dx +A\bigg(\int_{\R^2} \phi^2 dx \bigg)^2.
\label{Ham00}
\end{align}

\medskip

\item[(ii)]
Accordingly, we have the following concentration estimate: given any $\eta,  \eps>0$, 
\begin{align*}
\lim_{L\to \infty}\rho_L\Big( \big\{ \phi\in \dot{H}^{-\eta}(\T^2_L):  \| \phi\|_{\dot{H}^{-\eta}(\T^2_L)} \ge \eps   \big\} \Big)=0. 
\end{align*}

\noi
As a consequence, the infinite volume limit as $L\to \infty$ in the sense of weak convergence, is the trivial measure $\dl_0$ placing unit mass on the zero field,  which corresponds to the minimizer of the Hamiltonian \eqref{Ham00} generating the grand-canonical $\Phi^3_2$-measure.

\end{itemize}

\end{theorem}
As mentioned above, the question raised by McKean \cite{McKean} and Rider \cite{Rider} is to identify the $\infty$-volume Gibbs states for one-dimensional focusing Gibbs measures.   Our result in Theorem~\ref{THM:00} extends their work to a setting where small-scale (ultraviolet) issues arise.

Notice that the scaling $L^4$ in the free energy limit \eqref{free0} is not of volume order $L^2=\T_{L}^2$ since we are considering a large deviation problem with unbounded fields $\phi$.  The main contribution to the measure comes from atypical fields that are highly concentrated and peaked. Consider a function of the form
\begin{align*}
\phi_L(\cdot)=L^2\phi(L\cdot),
\end{align*}

\noi
with height $L^2$ and width $\frac 1L$,  a scaled ``soliton''. Under this scaling, the kinetic and potential energy terms are balanced, making the total energy scale like $L^4$
\begin{align*}
H(\phi_L)=L^4H(\phi) \sim L^4,
\end{align*}

\noi
rather than the volume order $L^2$.  This hints at the $L^4$-scaling in the asymptotic behavior of the free energy in \eqref{free0}.
This scaling is also crucial for controlling the error term $O(L^\al)$ arising from the infrared (large-scale) divergence
\begin{align}
\log Z_L \approx -L^4 \inf_{\phi \in H^1(\R)} H(\phi)+O(L^\al)
\label{scall4}
\end{align}

\noi
where $0<\al<4$. The translation invariance of the measure $\rho_L$ implies that the location of the scaled profile with width $\frac 1L$ is uniformly distributed over the torus. By combining these observations, as $L\to \infty$ we conclude that the measure converges weakly to a delta measure $\dl_0$ supported on the zero path.



\begin{remark} \rm 
A non-volume order scaling $L^3$ for the free energy limit appears in Rider's model \cite{Rider} for the one-dimensional focusing $|\phi|^4$ measure under Brownian bridge, that is, with interaction $\s|\phi|^4$, $\s<0$.  In that case, the expected size of the quartic term $\int_{\T_L}|\phi|^4 dx$ scales like  $L^3$,  which is much larger than the volume order $L=\T_L$.  Nonetheless, Rider obtains the same non-zero limiting free energy density like \eqref{free0}, as the leading contribution comes from highly concentrated paths.
\end{remark}


\begin{remark}\rm 
By choosing a rescaling $L^{-\g}$ that balances the contributions of both terms in \eqref{GGibs1}, one can consider the measure with
\begin{align}
\frac \s3  \int_{\T^2_L}   :\! \phi^3 \!: \, dx+\frac{A}{L^\g}\bigg(\int_{\T^2_L}  :\! \phi_{L}^2 \!: \, dx   \bigg)^2
\label{resc0}
\end{align}

\noi
for some $\g>0$.  However, if we add a decaying factor $L^{-\g}$ in front of the quartic term as in \eqref{resc0}, the partition function diverges as $L\to \infty$; see Remark 4.2 in \cite{OSeoT}.  For the size of each centered random variable in \eqref{resc0}, see Remark \ref{REM:size}.

\end{remark}





\subsection{Main result}

In this subsection, we state our main theorem \ref{THM:0}. We first provide an overview of the $L$-periodic problem on the dilated torus $\T^2_L$  and introduce the relevant notation.


Given $L>0$, we denote by $\T^2_L=(\R / L \Z)^2$  the dilated torus. Let us also define 
\begin{align*}
\Z^2_L=(\Z/L)^2.
\end{align*}

\noi
For any given $\ld \in \Z^2_L$, we define
\begin{align}
e^L_\ld(x)=\frac{1}{L}e^{2\pi i \ld \cdot x}
\label{el}
\end{align}

\noi
for $x \in \T^2_L$. Note that $\{e^L_L\ld \}_{\ld \in \Z^2_L}$ is an orthonormal basis of $L^2(\T^2_L)$. For any $\ld \in \Z^2_L$,  the Fourier transform $\ft f(\ld)$ of a function $f$ on $\T^2_L$ is defined by 
\begin{align*}
\ft f(\ld)=\int_{\T^2_L} f(x) \cj{e^L_\ld(x)} dx,
\end{align*}

\noi
with the corresponding Fourier representation:
\begin{align*}
f(x)=\sum_{\ld \in \Z^2_L} \ft f(\ld)e^L_\ld(x).
\end{align*}

\noi
We now review the construction of the $\Phi^3_2$-measure on $L$-periodic distributions on $\T^2_L$, namely, $\mathcal{D}'(\T^2_L)$.

Let $\mu_L$ denote a Gaussian measure on $\mathcal{D}'(\T^2_L)$, formally defined by
\begin{align*}
d\mu_{L}(\phi)&=Z_{L}^{-1}e^{-\frac 12 \| \phi \|^2_{H^1(\T^2_L)}  } \prod_{x\in \T^2_L}d\phi(x)\\
&=Z_{L}^{-1}\prod_{\ld \in \Z^2_L} e^{-\frac{1}{2} \jb{\ld}^2 |\ft \phi(\ld)|^2 }d\ft \phi(\ld)
\end{align*}

\noi
where $\jb{\cdot} = (1+|\cdot|^2)^\frac{1}{2}$, and $\ft \phi(\ld)$, $\lambda \in \mathbb{Z}^2_L$, represents the Fourier transform of $\phi$ on $\mathbb{T}^2_L$.  The measure $\mu_L$ corresponds to the massive Gaussian free field on $\T^2_L$, defined as the law of the following Gaussian Fourier series
\begin{align}
\o\in \O \longmapsto u_L(x;\o)=\sum_{\ld \in \Z^2_L} \frac{g_{L\ld}(\o) }{\jb{\ld}}e^L_\ld \in \mathcal{D}'(\T^2_L).
\label{RFOU}
\end{align}

\noi
Here, $\{ g_n \}_{n \in \Z^2}$ is a sequence of mutually independent standard complex-valued Gaussian random variables on a probability space $(\O,\F,\PP)$ 
conditioned on  $g_{-n} = \cj{g_n}$ for all $n \in \Z^2$. Denoting the law of a random variable $X$ by $\Law(X)$  (with respect to the underlying probability measure $\PP$), we have
\begin{align*}
\Law_{\PP} (u_L) = \mu_L
\end{align*}

\noi
for $u$ in \eqref{RFOU}. For any $L>0$,  $\mu_L$ is supported on $H^{s}(\T^2_L)\setminus L^2(\T^2_L)$ when $s<0$.

\begin{remark}\rm
For technical considerations, we employ a massive Gaussian free field as our reference measure. That is, we introduce an identity ``mass" term into the covariance $(1-\Delta)^{-1}$ to avoid the degeneracy of the zeroth Fourier mode. If one wishes to consider the massless Gaussian free field, it is necessary to restrict discussion to fields which satisfy the mean-zero condition.
\end{remark}

As is usual for fields based on the Gaussian free field in higher dimensions, attention must be given to ultraviolet (small scale) divergences. 
To explain this problem, let $N \in \N$ and define the frequency projector $\P_N $ onto the frequencies $\{ |\ld|\le N\}$ as follows
\begin{align}
\P_N f=\sum_{|\ld|\le N}\ft f(\ld) e^L_\ld.
\label{defpr}
\end{align}

\noi
We set $f_N:=\P_N f$. Letting $L>0$ and $\phi$ be the free field under measure $\mu_L$, it follows from \eqref{RFOU} and \eqref{el}, and a Riemann sum approximation that 
\begin{align}
\<tadpole>_{L,N}&:=\E_{\mu_L}\Big[|\P_N\phi(x)|^2 \Big]=\sum_{\substack{\ld \in \Z^2_L\\ |\ld|\le N}} \frac{1}{\jb{\ld}^2} \frac{1}{L^2} \notag \\
&=\sum_{\substack{n\in \Z^2 \\ |n|\le LN}} \frac{1}{\jb{\frac{n}{L}}^2}\frac{1}{L^2} \sim \int_{\R^2} \ind_{\{ |y|\le N  \}} \frac{dy}{1+|y|^2} \sim \log N \to \infty
\label{Ultra1}
\end{align}

\noi
as $N\to \infty$, independently of $x \in \T^2_L$ thanks to the stationarity of the Gaussian free field $\mu_L$. In particular, $\phi = \lim_{N\to \infty} \P_N \phi$ is merely a distribution, meaning that the expression $( \P_N u)^k$, where $k \geq 2$, does not converge to any limit.  Hence, for each $x\in \T^2_L$, we define the Wick powers $:\! \phi_{N}^2 \!:$ and  $:\! \phi_{N}^3 \!: $ as follows 
\begin{align}
:\! \phi_{N}^2 \!: \, &= \phi_{N}^2 - \<tadpole>_{L,N} \label{Wickqu}\\
:\! \phi_{N}^3 \!: \, &= \phi^3_{N} -3\,\<tadpole>_{L,N} \phi_{N}. \label{Wickcu}
\end{align}

\noi
One can show that $:\! \phi_{N}^2 \!: \,$ and $:\! \phi_{N}^3 \!: \,$ converge, almost surely and in $L^p(\O)$ for any finite $p \ge 1$ as $N\to \infty$, to limits which we denote by $:\! \phi^2 \!: \,$ and $:\! \phi^3 \!: \,$ in $H^s(\T^2_L)$, where $s<0$. 
We study the corresponding renormalized interaction potential
\begin{align}
\textbf{V}_{N}^L(\phi):=\frac \s3  \int_{\T^2_L}   :\! \phi_{N}^3 \!: \, dx+A\bigg(\int_{\T^2_L}  :\! \phi_{N}^2 \!: \, dx   \bigg)^2
\label{Den1}
\end{align}

\noi
where $\s \in \R\setminus\{0\}$ and $A\ge 1$.
We define the renormalized truncated Gibbs measure
\begin{align}
d\rho_{N,L}(\phi) = Z_{L,N}^{-1} \exp\Big\{ -\textbf{V}_N^L(\phi)   \Big\} d\mu_{L}(\phi)
\label{Gibbs7}
\end{align}

\noi
with the partition function $Z_{L,N}$
\begin{align}
Z_{L,N}=\int e^{-\frac{\s}{3} \int_{\T^2_L} :\, \phi_N^3 :\, dx  -A\big( \int_{\T^2_L} :\, \phi_N^2: \, dx   \big)^2 }  d \mu_L(u).
\end{align}

The following proposition shows that the objects just defined converge as the frequency cutoff $N$ goes to $\infty$.

\begin{proposition}\label{PROP:N}
Let $L>0$ and $\s \in \R \setminus\{0\}$. Given any finite $ p \ge 1$,   $\textbf{V}_{N}^L(\phi) $ in \eqref{Den1} converges in $L^p(d\mu_L) $  as $N\to \infty$, to a limit $\textbf{V}^L(\phi)$,
\begin{align}
\textbf{V}^L(\phi)=\frac{\s}{3} \int_{\T^2_L} :\, \phi^3 :\, dx +A\bigg( \int_{\T^2_L} :\, \phi^2: \, dx   \bigg)^2.
\label{IP0}
\end{align}

\noi 
Moreover, there exists $A_0 \ge 1$ and $C_{p, A_0} > 0$ such that 
\begin{equation}
\sup_{N\in \N} \Big\| e^{-\textbf{V}_{N}^L(\phi)}\Big\|_{L^p(d\mu_L)}
\leq C_{p, A_0}  < \infty
\end{equation}

\noi
for any $A\ge A_0$. In particular,  we have
\begin{equation}
\lim_{N\rightarrow\infty}e^{-\mathbf{V}_{N}^L(\phi)}=e^{-\mathbf{V}^L(\phi)}
\qquad \text{in } L^p(d\mu_L).
\label{conv1}
\end{equation}

\noi
As a consequence, the truncated renormalized $\Phi^3_{2}$-measure in \eqref{Gibbs7} converges, in the sense of \eqref{conv1}\footnote{This implies that the truncated measure $\rho_{N,L}$ converges in total variation to the limiting measure $\rho_L$}, to the $\Phi^3_{2}$-measure given by
\begin{align}
d\rho_{L}(\phi)&= Z_L^{-1} e^{-\mathbf{V}^L(\phi)}d\mu_L(\phi)
\label{Gibbs8}
\end{align}

\noi
where $Z_L$ is the partition function
\begin{align}
Z_L=\int e^{-\textbf{V}^L(\phi)}d\mu_L(\phi).
\label{PART0}
\end{align}

\noi
Furthermore, for each $0<L<\infty$, the limiting $\Phi^3_{2}$-measure $\rho_L$ is mutually absolutely continuous with the base Gaussian measure $\mu_L$.	
\end{proposition}

Proposition \ref{PROP:N} shows that taking proper renormalizations on the interaction potential gives the control of the ultraviolet (small scale) issues. 

Before presenting the main result (Theorem \ref{THM:0}), we explain the infrared (large scale) divergence as $L\to \infty$. Proposition \ref{PROP:N} shows that $\rho_L$ and $\mu_L$ are mutually absolutely continuous for each finite $L>0$. However, this equivalence between $\rho_L$ and $\mu_L$ is not uniform as $L \to \infty$. This lack of uniformity arises because the potential energy $\textbf{V}^L(\phi)$, which is the limit of $\textbf{V}^L_N$ as defined in \eqref{Den1}, has polynomial growth
\begin{align*}
\textbf{V}^L(\phi) \sim L^2
\end{align*}

\noi
under $\mu_L$ as $L\to \infty$. This indicates that any possible infinite-volume limit $\rho_\infty$ on $\R^2$ and the base Gaussian measure $\mu_\infty$\footnote{Namely, the large torus limit of $\mu_L$} are mutually singular. See Lemma \ref{LEM:Cor00} (ii). This makes it nontrivial to get the uniform control of  the $L$-periodic $\Phi^3_2$-measure and is the main issue in the study of the infinite volume limit as $L\to \infty$.


The main contribution of this paper is to exhibit concentration of the $L$-periodic $\Phi^3_2$-measure around zero, which is the unique minimizer of Hamiltonian \eqref{HamR2} as $L\to \infty$ in the range of parameters we consider.


\begin{theorem}\label{THM:0}
Given any $\s\in \R \setminus\{0\}$,  there exists a constant $A_0=A_0(\s) \ge 1$ independent of $L \ge 1$ such that for  all $A \ge A_0$, the free energy $\log Z_L$ of the grand-canonical partition function $Z_L$ in \eqref{PART0} satisfies
\begin{align}
\lim_{L\to \infty} \frac{\log Z_{L}}{L^4}=-\inf_{\phi\in H^1 (\R^2)} H(\phi)
\label{Free}
\end{align}

\noi
where
\begin{align}
H(\phi)&=\frac 12\int_{\R^2} |\nb \phi|^2 dx+\frac{\s}{3} \int_{\R^2} \phi^3 dx +A\bigg(\int_{\R^2} \phi^2 dx \bigg)^2.
\label{HamR2}
\end{align}

\noi
Moreover, if $\rho_L$ is the grand-canonical $\Phi^3_2$ measure \eqref{Gibbs8} on finite volume $\T^2_L$, associated with the Hamiltonian
\begin{align}
H_L(\phi)&=\frac 12\int_{\T^2_L} |\nb \phi|^2 dx+\frac{\s}{3} \int_{\T^2_L} \phi^3 dx +A\bigg(\int_{\T^2_L} \phi^2 dx \bigg)^2,
\label{HamL}
\end{align}

\noi
then, given any $\eta, m, \eps>0$ and test functions\footnote{We extend the test functions to $\R^2$ by periodic extension.} $g_j$ with $\supp(g_j)\subset \T^2_L$,
\begin{align}
\lim_{L\to \infty}\rho_L\Big( \big\{ \phi\in H^{-\eta}(\T^2_L): \max_{1\le j \le m} \big|\jb{\phi,g_j} \big| \ge \eps  \big\} \Big)=0
\label{Conc}
\end{align}

\noi
for all $\s \in \R \setminus \{0\}$ and all $A \ge A_0$.  As a consequence, the infinite volume limit as $L\to \infty$, in the sense of weak convergence, 
\begin{align*}
\rho_L \too \dl_0 
\end{align*}

\noi
is the trivial measure $\dl_0$ that places unit mass on the zero field.

\end{theorem}

The unboundedness of the cubic interaction $\frac{\s}{3} \int_{\T^2_L } \phi^3 dx $ results in a sharp local concentration of the field around a single minimizer of the Hamiltonian \eqref{HamR2}  as $L \to \infty$, which is zero when $A$ is sufficiently large. This collapse is a result of the intense competition between the cubic interaction $\frac \s3 \int \phi^3 dx$, which drives the ground state energy towards $-\infty$, and the taming by the (Wick-ordered) $L^2$-norm $A\big(\int \phi^2 dx\big)^2$, acting to counterbalance the focusing nature. As long as the chemical potential $A$ is sufficiently large, the unboundedness of the cubic interaction can be controlled by the taming part. See Remark \ref{GNS} for an explanation of the critical value of the chemical potential $A$. We also point out that compared to the $\Phi^4$ theory whose infinite volume limit depends qualitatively on the temperature parameter $\beta$, all results in Theorem \ref{THM:0} are true regardless of temperature scale for the temperature dependent ensemble $e^{-\be H(\phi)} \prod\limits_{x} d\phi(x)$. In other words,  we do not encounter a change of phase depending on low and high temperatures. In Theorem \ref{THM:0}, the infinite volume limit is not only unique but is in fact trivial for every temperature.





Our method is based on \cite{BG, BG1}. The first step in proving the concentration result \eqref{Conc} is to establish a large deviation estimate, in other words, to compute the first order behavior of the free energy $\log Z_L$ in the limit $L\to \infty$ \eqref{Free}.  In contrast to the one-dimensional case, where the ensemble is supported on a space of functions, the $\Phi^3_2$-measure on the finite  volume $\T^2_L$ lives on the space of distributions on $\T^2_L$. Because of the renormalization required by this low regularity, one cannot proceed with the computation of the free energy as in the one-dimensional focusing $\Phi^4_1$-measure treated in \cite{Rider, TW}, since the renormalization process destroys the coercive structure. In particular, the main task of our work is to show that the free energy $\log Z_L$ in the infinite volume limit $L\to\infty$ is ultraviolet stable, namely, the limit $L\to \infty$ is uniform in $N\ge 1$, where $N$ is the ultraviolet cutoff parameter. To achieve this, we initially address the small-scale singularities and extend the variational characterization of the free energy without the small-scale (ultraviolet) cutoff, using  Gamma convergence. Then we control large scale (infrared) divergences as $L\to \infty$, arising from the stationarity of the $\Phi^3_2$-measure. In particular, as pointed out above \eqref{scall4}, the nonvolume order scaling $L^4$ is essential for controlling the error term $O(L^\al)$, $0<\al<4$, resulting from the infrared divergence.


\begin{remark} \rm 
Theorem \ref{THM:0} also holds when the Gibbs measure in \eqref{Gibbs8}, with the massive Gaussian field as the base field, is replaced by the one, with the massless Gaussian field as the base field.
\end{remark}

\begin{remark} \rm
The critical value $A^*$ of the chemical potential $A$ is related to the ground state $Q$, that is, the minimizer of the Hamiltonian \eqref{Ham0}, under the $L^2$ mass constraint. See Remarks \ref{REM:GNS} and \ref{REM:CRIT01} for an explanation of the critical chemical potential. 
Notice that $A_0$ chosen in Proposition \ref{PROP:N} and Theorem \ref{THM:0} is taken to be sufficiently large, that is, $A_0\gg A^*$.  
It would be interesting to investigate whether the measure can still be constructed, and whether Theorem \ref{THM:0} remains valid at and around the critical value $A^*$.
\end{remark}




\begin{remark}\rm \label{REM:not}
The $\Phi^3_2$ measure \eqref{Gibbs8} on the finite volume $\T^2$ is the invariant measure for the dynamical parabolic and hyperbolic $\Phi^3$-model on $\T^2\times \R_{+}$ 
\begin{align}
\dt u -  \Dl  u + :\!u^{2}\!: +A\bigg(\int_{\T^2} :\! u^2\!: dx \bigg) u  &=  \sqrt {2}\xi  \label{SNLH} \\
\dt^2 u+\dt u -  \Dl  u +:\! u^{2} \!: +A\bigg(\int_{\T^2} :\! u^2\!: dx \bigg)u &=  \sqrt {2}\xi \label{SNLW},
\end{align} 


\noi
where $\xi=\xi(x,t)$ denotes the space-time white noise on $\T^2 \times \R_+$. We point out that the Gibbs measure \eqref{Jaffe1} constructed by adding an $L^2$ cut-off is not suitable to generate any Schr\"odinger\,/\,wave\,/\,heat dynamics since (i)~the renormalized cubic power $:\! \phi^3 \!:$ makes sense only  in the real-valued setting and hence is not suitable for the Schr\"odinger equation with complex-valued solution and (ii)~\eqref{SNLH} and \eqref{SNLW} do not preserve the $L^2$-norm of a solution and thus are incompatible with the Wick-ordered $L^2$-cutoff.

\end{remark}

\subsection{Organization of the paper} 
In Section \ref{SEC:Not}, we introduce some notations and preliminary lemmas. Section \ref{SEC:Varc} presents the variational characterization of the minimizers of the Hamiltonian. In Section \ref{SEC:ULTS}, we establish ultraviolet stability for the $\Phi^3_2$-measure by using the variational formulation of the partition. Section \ref{SEC:Free} analyzes the behavior of the free energy $\log Z_L$ as $L \to \infty$. Finally, in Section \ref{SEC:PRMA}, we prove the main results, specifically Theorem \ref{THM:0}.

\section{Notations and basic lemmas}
\label{SEC:Not}

When addressing regularities of functions and distributions, we  use $\eta > 0$ to denote a small constant. We usually  suppress the dependence on such $\eta > 0$ in estimates. For $a, b > 0$, $a\lesssim b$  means that
there exists $C>0$ such that $a \leq Cb$. By $a\sim b$, we mean that $a\lesssim b$ and $b \lesssim a$. Regarding space-time functions, we use the following short-hand notation $L^q_TL^r_x$ = $L^q([0, T]; L^r(\T^2))$, etc.

\subsection{Function spaces}
\label{SUBSEC:21}

Let $s \in \R$ and $1 \leq p \leq \infty$. We define the $L^p$-based Sobolev space $W^{s, p}(\T^2_L)$ by 
\begin{align*}
\| f \|_{W^{s, p}(\T^2_L)} = \big\| \F^{-1} [\jb{\ld}^s \ft f(\ld)] \big\|_{L^p(\T^2_L)}.
\end{align*}

\noi
When $p = 2$, we have $H^s(\T^2_L) = W^{s, 2}(\T^2_L)$.

Let $\phi:\R \to [0, 1]$ be a smooth  bump function supported on $[-\frac{8}{5}, \frac{8}{5}]$ 
and $\phi\equiv 1$ on $\big[-\frac 54, \frac 54\big]$.
For $\xi \in \R^d$, we set $\varphi_0(\xi) = \phi(|\xi|)$
and 
\begin{align}
\varphi_{j}(\xi) = \phi\big(\tfrac{|\xi|}{2^j}\big)-\phi\big(\tfrac{|\xi|}{2^{j-1}}\big)
\label{phi1}
\end{align}

\noi
for $j \in \N$.
Then, for $j \in \Z_{\geq 0} := \N \cup\{0\}$, 
we define  the Littlewood-Paley projector  $\pi_j$ 
as the Fourier multiplier operator with a symbol $\varphi_j$.
Note that we have 
\begin{align*}
\sum_{j = 0}^\infty \varphi_j (\xi) = 1
\end{align*}

\noi
for each $\xi \in \R^2$ and $f = \sum_{j = 0}^\infty \pi_j f$. We next recall the basic properties of the Besov spaces $B^s_{p, q}(\T^2_L)$
defined by the norm
\begin{equation*}
\| u \|_{B^s_{p,q}(\T^2_L)} = \Big\| 2^{s j} \| \pi_{j} u \|_{L^p_x(\T^2_L)} \Big\|_{\l^q_j(\Z_{\geq 0})}.
\end{equation*}

\noi
We denote the H\"older-Besov space by  $\mathcal{C}^s (\T^2_L)= B^s_{\infty,\infty}(\T^2_L)$.
Note that the parameter $s$ measures differentiability and $p$ measures integrability. In particular, $H^s (\T^2_L) = B^s_{2,2}(\T^2_L)$
and for $s > 0$ and not an integer, $\mathcal{C}^s(\T^2_L)$ coincides with the classical H\"older spaces $C^s(\T^2_L)$; see  \cite{Gra2}.

\section{Variational characterization of the minimizers}\label{SEC:Varc}

In this section, we investigate the stability of minimizers for the Hamiltonian \eqref{H2}. To analyze stability, we begin by examining the Gagliardo-Nirenberg-Sobolev inequality.

\subsection{Gagliardo-Nirenberg-Sobolev inequality}

The Gagliardo-Nirenberg-Sobolev (GNS) inequality plays an important role in the study of the the minimizers of the Hamiltonian
\begin{align}
H_{L}(\phi)=\frac 12\int_{\T^2_L} |\nb \phi|^2 dx+\frac{\s}{3} \int_{\T^2_L} \phi^3 dx +A\bigg(\int_{\T^2_L} \phi^2 dx \bigg)^2
\label{H2}
\end{align}

\noi
for any $1\le L \le  \infty$. When $L=\infty$, the Hamiltonian is defined for functions of the full space $\T^2_{\infty}=\R^2$.
The following result on the optimal constant $C_{\text{GNS}}$ and optimizers was proved by Weinstein \cite{WEST} for general dimensions $d\ge 2$. We present the case $d=2$.  
\begin{proposition} \label{prop: W}
For any finite $p >2$ and $\phi \in H^1(\R^2)$, we have
\begin{align}
\|\phi \|_{L^p(\R^2)}^p \le C_{\textup{GNS}}(p)\| \nb \phi \|_{L^2(\R^2)}^{p-2} \| \phi \|_{L^2(\R^2)}^{2}
\label{GNS00}
\end{align}

\noi
where 
\begin{align*}
C_{\textup{GNS}}^{-1}(p):=\inf_{\substack{\phi \in H^1(\R^2)\\ \phi \neq 0}} \frac{\| \nb \phi \|^{p-2}_{L^2(\R^2)} \| \phi \|_{L^2}^2(\R^2)}{\|\phi \|_{L^p}^p(\R^2)}.
\end{align*}

\noi
Then, the minimum is attained at a positive, radial, and exponentially decaying function $Q\in H^1(\R^2)$  which is the unique radial solution to the elliptic equation on $\R^2$ 
\begin{align*}
(p-2)\Dl Q +2 Q^{p-1} -2Q=0
\end{align*}

\noi
with the minimal $L^2$-norm (namely, the ground state). In particular, we have
\begin{align*}
C_{\textup{GNS}}(p)=\frac{p}{2}\| Q\|_{L^2(\R^2)}^{2-p}.
\end{align*}

\end{proposition}

The GNS inequality \eqref{GNS00} fails on the bounded domain $\T^d_L$. For example, \eqref{GNS00} does not hold for constant functions. A related inequality, with an additional term on the right, does hold on $\T^d_L$ and appears below in \eqref{GNST}. The result is elementary, but we could not locate a proof in a form suitable for our application.

\begin{lemma}\label{LEM:GNST}
Let $2<p<\infty$ if $d=1,2$ and $2<p<\frac{2d}{d-2}$ if $d \ge 3$. Then, there exists a constant $C=C(d,p)$ independent of $L$ such that for any $\phi \in H^1(\T^d_L)$
\begin{align}
\| \phi \|_{L^p(\T^d_L)} \le C \|\nb \phi \|_{L^2(\T^d_L)}^{\dr }\| \phi \|_{L^2(\T^d_L)}^{(1-\dr)}+CL^{-\dr}\| \phi\|_{L^2(\T^d_L)}.
\label{GNST}
\end{align}

\noi
where $\dr=d(\frac{1}{2}-\frac{1}{p})$.

\end{lemma}

\begin{proof}

We first assume the case $L=1$, namely, for any $\varphi \in H^1(\T^d)$  
\begin{align}
\| \varphi \|_{L^p(\T^d)} \le C \|\nb \varphi \|_{L^2(\T^d)}^{\dr }\| \varphi \|_{L^2(\T^d)}^{(1-\dr)}+C\| \varphi\|_{L^2(\T^d)}
\label{bbca}
\end{align}

\noi
where $\dr=d(\frac{1}{2}-\frac{1}{p})$, and then prove the main result \eqref{GNST}. For any $\phi \in H^1(\T^d_L)$ and $1\le L<\infty$, we set $\phi_L(x):=L^{\frac{d}{p}}\phi(Lx)$. Then, $
\phi_L\in H^1(\T^d)$ and  
\begin{align*}
\| \nb \phi_L \|_{L^2(\T^d)}&=L^{\frac{d}{p}-\frac{d}{2}+1}\| \nb \phi \|_{L^2(\T^d_L)}\\
\| \phi_L \|_{L^2(\T^d)}&=L^{\frac{d}{p}-\frac{d}{2}}\|  \phi \|_{L^2(\T^d_L)}.
\end{align*}

\noi
By using \eqref{bbca}, we have 
\begin{align*}
\| \phi \|_{L^p(\T^d_L)}=\| \phi_L \|_{L^p(\T^d)}&\le C \| \nb \phi_L \|^{\dr}_{L^2(\T^d)} \| \phi_L \|_{L^2(\T^d)}^{1-\dr}+C\| \phi_L \|_{L^2(\T^d)}\\
&\le C L^{\dr(\frac{d}{p}-\frac{d}{2}+1)+(1-\dr)(\frac{d}{p}-\frac{d}{2}) } \| \nb \phi \|^{\dr}_{L^2(\T^d_L)} \| \phi \|_{L^2(\T^d_L)}^{1-\dr}+CL^{-\dr}\| \phi \|_{L^2(\T_L^d)}\\
&\le C  \| \nb \phi \|^{\dr}_{L^2(\T^d_L)} \| \phi \|_{L^2(\T^d_L)}^{1-\dr}+CL^{-\dr}\| \phi \|_{L^2(\T_L^d)}.
\end{align*}

\noi
Hence, it suffices to prove \eqref{bbca}. By interpolation in $L^p$, we have that for any $u\in H^1(\T^d)$ 
\begin{align}
\|u \|_{L^p(\T^d)} \le \| u\|_{L^2(\T^d)}^{1-\dr} \| u \|_{L^r(\T^d)}^{\dr}
\label{intp0}
\end{align}

\noi
where $\frac{1}{p}=\frac{1-\dr}{2}+ \frac{\dr}{r}$ and $2<p<r<\infty $ if $d=1,2$, and $2<p<r\le \frac{2d}{d-2}$ if $d \ge 3$. Also, there exists an extension operator $E$ from $H^1(\T^d)$ to $H^1(\R^d)$ and a constant $C$ such that for every $u\in H^1(\T^d)$, $Eu=u $ on $\T^d$ and $\supp Eu \subset \T^d_{L_0}$ for some $L_0 \gg 1 $ and 
\begin{align}
\| Eu  \|_{H^1(\R^d)} \le C \| u\|_{H^1(\T^d)}.
\label{bb1}
\end{align}

\noi
Since $Eu=u$ on $\T^d$, using the Sobolev inequality, and \eqref{bb1}, we have
\begin{align}
\| u \|_{L^r(\T^d)} \le \|Eu \|_{L^r(\R^d)} &\le C \| Eu \|_{H^1(\R^d)} \notag \\
&\le C \| u \|_{H^1(\T^d)}
\label{intp1}
\end{align}

\noi
where $\frac{1}{r}=\frac{1}{2}-\frac{1}{d}$. Combining \eqref{intp0} and \eqref{intp1}, we have
\begin{align*}
\|u \|_{L^p(\T^d)} \le C \| u\|_{L^2(\T^d)}^{1-\dr} \| \nb u \|_{L^2(\T^d)}^{\dr}+C\|  u \|_{L^2(\T^d)},
\end{align*}

\noi
which completes the proof of \eqref{bbca}. 
\end{proof}

\begin{remark}\rm \label{REM:GNS}

The sharp Gagliardo-Nirenberg-Sobolev (GNS) inequality on $\R^2$ 
\begin{align}
\|\phi \|_{L^3(\R^2)}^3\le C_{\text{GNS}} \|\nb \phi \|_{L^2(\R^2)} \| \phi \|_{L^2(\R^2)}^2
\label{GNS}
\end{align}

\noi
plays an important role in the study of the $\Phi^3_2$-measure. The positive radial solution to the following semilinear elliptic equation on $\R^2$
\begin{align}
\Dl Q+2Q^2-2Q=0
\label{ellip}
\end{align}

\noi
appearing in Proposition \ref{prop: W} is referred to as the ground state for the associated elliptic problem \eqref{ellip}.


The construction of the $\Phi^3_2$-measure \eqref{Gibbs8} and relevance of the GNS inequality \eqref{GNS} can be seen at heuristic level by formally rewriting \eqref{PART0} as a functional integral (ignoring the renormalization) 
\begin{align}
Z_{L}&=\int e^{-\frac \s3 \int \phi^3-A\big( \int \phi^2 dx \big)^2 } e^{-\frac 12 \int |\nb \phi|^2 dx} \prod_{x \in \T^2_L} d\phi(x)
\label{GNSZL}
\end{align}

\noi
for $\s \in \R\setminus\{0\}$ and $A>0$. Thanks to the GNS inequality \eqref{GNS} and Young's inequality, we can control the cubic interaction as follows
\begin{align*}
\|\phi \|_{L^3(\R^2)}^3 \le \dl \| \nb \phi \|_{L^2}^2 +c(\dl) \| \phi \|_{L^2}^4
\end{align*}

\noi
for all sufficiently small $\dl>0$ and some large constant $c(\dl)$ depending on $\dl$ and $C_{\text{GNS}}$ in \eqref{GNS}. From this, we can establish an upper bound
\begin{align*}
\eqref{GNSZL} \le \int e^{-(A-c(\dl)) \big( \int \phi^2 dx \big)^2   } e^{-\big(\frac 12-\dl\big) \int |\nb \phi|^2 dx} \prod_{x \in \T^2_L} d\phi(x).
\end{align*}

\noi 
Hence, when the chemical potential $A$ is sufficiently large, we expect the partition function $Z_{L}$ to be finite.  In fact, the choice of the exponent $\g=2$ in $A\Big(\int \phi^2 dx   \Big)^\g$ with $A\gg 1$ is optimal. When $\g<2$ or when $\g=2$ and $A$ is sufficiently small, the taming by the Wick-ordered $L^2$-norm in \eqref{GNSZL} is too weak to control the cubic interaction, and thus we expect an nonnormalizability result to hold. See \cite{OSeoT} for a rigorous argument. The optimal threshold for $A$ when $\gamma=2$ is related to the ground state $Q$, given that $c(\dl)$ depends on $C_{\text{GNS}}$. See also Lemma \ref{LEM:MIN1} (i). It  would be interesting to see whether the $\Phi^3_2$-measure can be constructed as a probability measure at this optimal threshold, even on the finite volume $\T^2_L$.

\end{remark}

\subsection{Existence and stability of minimizers}

In this subsection, we study the optimizers for the Hamiltonian \eqref{H2}, along with their stability properties.

We first define the following Hamiltonian, which does not include taming by the $L^2$-norm:
\begin{align}
H_0(\phi)=\frac{1}{2} \int_{\R^2} |\nb \phi|^2 dx +\frac{\s}{3} \int_{\R^2 } \phi^3 dx
\label{H0}
\end{align}

\noi
for any $\s \in \R\setminus\{0\}$. For any fixed $q >0$, define
\begin{align}
 H_{0, q}^{*}=\inf_{\phi \in H^1 (\R^2) } \big\{ H_0(\phi): M(\phi)=q  \big\}
\label{H0q}
\end{align}

\noi
where 
\begin{align}
M(\phi)=\int_{\R^2 } \phi^2 dx.
\label{mass}
\end{align}

We first prove the following lemma.

\begin{lemma}\label{LEM:Min}
For every $q>0$, we have
\begin{align*}
-\infty<H_{0, q}^{*}<0
\end{align*}

\noi
where $H_{0, q}^{*}$ is given as in \eqref{H0q}.

\end{lemma}

\begin{proof}
We first assume $\s>0$. Take any function $W \in H^1(\R^2)$ such that $M(W)=q$, $W>0$, and so $\int W^3 dx>0$. For each $\zeta>0$, define $W_\zeta(x)=-\zeta W(\zeta x)$. Then, we have
$M(W_\zeta)=M(W)=q$ for every $\zeta>0$, where $M(W)$ is as in \eqref{mass}. Moreover, we get
\begin{align*}
H_0(W_\zeta)=\frac {\zeta^2}2\int_{\R^2} |\nb W|^2 dx -\frac{\s \zeta}{3} \int_{\R^2} W^3 dx.
\end{align*}

\noi
Hence, by choosing $\zeta$ sufficiently small, we have $H_0(W_\zeta)<0$. From the definition of $H_{0, q}^{*}$, we obtain $H_{0, q}^{*}\le H(W_\zeta)<0$. If $\s <0$, then one can proceed similarly with $W_\zeta(x)=\zeta W(\zeta x)$.

We now prove the lower bound. By the GNS \eqref{GNS00} and Young inequalities, we have
\begin{align}
\|\phi \|^3_{L^3(\R^2)}&\le C_{\text{GNS}} \|\nb \phi \|_{L^2(\R^2)} \| \phi \|_{L^2(\R^2)}^2 \notag  \\
&\le \dl \| \nb \phi \|_{H^1}^2 +A(\dl) \|\phi \|_{L^2}^4
\label{GNS01}
\end{align}

\noi
for every $\dl>0$, where $A=A(\dl)$ is a large constant depending on $\dl>0$. It follows from \eqref{GNS01} and $M(\phi)=q$ that
\begin{align}
H_0(\phi)&=\frac{1}{2} \int_{\R^2} |\nb \phi|^2 dx +\frac{\s}{3} \int_{\R^2 } \phi^3 dx \notag \\
&\ge \Big(\frac 12-\dl_1 \Big) \|\nb \phi \|^2_{L^2}-cA(\dl_1)q^2 \ge -cA(\dl_1)q^2>-\infty 
\label{GNS02}
\end{align}

\noi
for some small $\dl_1>0$ and a constant $c>0$. In view of \eqref{GNS02}, we obtain $H_{0, q}^{*}>-\infty$ for any fixed $q>0$.

\end{proof}

We next prove the existence of minimizers for the variational problem in \eqref{H0q}.
The set of minimizers for the problem \eqref{H0q} is defined by
\begin{align*}
\M_q=\{  \phi \in H^1(\R^2): H_0(\phi)=H_{0,q}^{*} \quad \text{and} \quad M(\phi)=q  \}.
\end{align*}

\noi
A \emph{minimizing sequence} for $H_{0,q}^{*}$ is any sequence $\{ \varphi_n\}$ of functions in $H^1(\R^2)$ satisfying
\begin{align*}
M(\varphi_n)=q 
\end{align*} 

\noi
for every $n \ge 1$ and 
\begin{align*}
\lim_{n\to \infty} H_0(\varphi_n)=H_{0, q}^{*}.
\end{align*}

\begin{lemma}\label{LEM:var}
For every $q >0$, the set $\M_q$ is not empty. 

\end{lemma}

\begin{proof}
For the proof, see \cite{Frank}.
\end{proof}

We now study the optimizers for the Hamiltonian \eqref{H2} with a taming by the $L^2$-norm, along with their stability properties. 

\begin{lemma}\label{LEM:MIN1}
Let $\s \in \R\setminus\{0\}$. 
\begin{itemize}
\item [(i)]
The Hamiltonian 
\begin{align}
H(\phi)=\frac 12\int_{\R^2} |\nb \phi|^2 dx+\frac{\s}{3} \int_{\R^2} \phi^3 dx+A\bigg(\int_{\R^2}\phi^2 dx\bigg)^2.
\label{HAMR22}
\end{align}

\noi
has the unique minimizer $\phi=0$ if $A > |H_{0,1}^{*}|$ and infinitely many minimizers if $A = |H_{0,1}^{*}|$, where
\begin{align}
 H_{0, 1}^{*}=\inf_{\phi \in H^1 (\R^2) } \big\{ H_0(\phi): M(\phi)=1  \big\}.
\label{CRA0}
\end{align}

\noi
Here, $H_0$ is the Hamiltonian given in \eqref{H0}.

\item[(ii)]
There exists a large constant $A_0 \ge 1$ such that for every $A \ge A_0$ and every $L >0$, the Hamiltonian
\begin{align*}
H_L(\phi)=\frac 12\int_{\T^2_L} |\nb \phi|^2 dx+\frac{\s}{3} \int_{\T^2_L} \phi^3 dx+A\bigg(\int_{\T^2_L}\phi^2 dx\bigg)^2
\end{align*}

\noi
has the unique minimizer $\phi=0$. Furthermore, there exists a constant $c>0$ independent of $L$ such that
\begin{align}
H_L(\varphi) \ge \inf_{\phi \in H^1(\T^2_L)}H_L(\phi) + c\Big(\| \nb  \varphi \|_{L^2(\T^2_L)}^2+\|\varphi \|_{L^2(\T^2_L)}^4 \Big).
\label{JJJJ0}
\end{align} 

\noi
In other words, if the energy $H_L(\varphi)$ is close to the minimal energy $\inf_{\phi \in H^1(\T^2_L)} H_L(\phi)$, then $\varphi$ is close to the minimizer, namely the zero function $\varphi=0$.

\end{itemize}

\end{lemma}

\begin{proof}
We first prove part (i).
Start from the decomposition of the minimization problem:
\begin{align}
\inf_{\phi \in H^1(\R^2)}H(\phi)&=\inf_{q\ge 0} \bigg\{ \inf_{\substack{\phi \in H^1(\R^2)\\ \| \phi \|_{L^2}^2=q}} H_0(\phi)+Aq^2 \bigg\} \notag \\
&=\inf_{q\ge 0} \bigg\{  q^2\inf_{\substack{\phi \in H^1(\R^2)\\ \| \phi \|_{L^2}^2=1}} H_0(\phi)+Aq^2 \bigg\} \notag \\
&=\inf_{q \ge 0} \Big\{ q^2H_{0,1}^{*}+Aq^2 \Big\}.
\label{qq}
\end{align}

\noi
Given that Lemma \ref{LEM:Min} shows that  $-\infty<H_{0, 1}^{*}<0$, if $A>|H_{0, 1}^{*}|$, then the minimum is achieved at $q = 0$ in \eqref{qq}. This shows that $\phi=0$ is the unique minimizer.

If $A=|H_{0, 1}^{*}|$, then from \eqref{qq}, we have  $\inf_{\phi \in H^1(\R^2)}H(\phi)=0$. For any $q \ge 0$ and $x_0\in\R^2$,  define $Q_{q,x_0}:=qQ(q^{\frac 12}(\cdot-x_0) )$ where $\|Q \|_{L^2}^2=1$ and
\begin{align*}
H_0(Q)=\inf_{\|\phi\|_{L^2}^2=1} H_0(\phi)=H_{0, 1}^{*}
\end{align*}

\noi
where $H_0$ is the Hamiltonian given in \eqref{H0}. The existence of such $Q$ is guaranteed by Lemma \ref{LEM:var}. Then, since $\|Q_{q,x_0} \|_{L^2}^2=q$ and
\begin{align*}
H_0(Q_{q,x_0})=\frac{q^2}{2} \int_{\R^2} |\nb Q|^2 dx+\frac{q^2\s }{3}\int_{\R^2}Q^3 dx=q^2 H_{0,1}^{*},
\end{align*}

\noi
we obtain 
\begin{align*}
H(Q_{q,x_0})=q^2 H_{0,1}^{*}+Aq^2=0,
\end{align*}

\noi
which shows that $\{Q_{q,x_0}\}_{q \ge 0,x_0\in \R^2}$ is a set of infinitely many minimizers.

We next prove Part (ii). From the GNS inequality on $\T^2_L$ (Lemma \ref{LEM:GNST}) and Young's inequality, we have 
\begin{align}
H_L(\varphi)\ge \frac {1-\dl}2 \int_{\T^2_L} |\nb \varphi|^2 dx +(A-c(\dl)-c(L)) \bigg( \int_{\T^2_L} \varphi^2 dx \bigg)^2 \ge 0
\label{JJJJ00}
\end{align}

\noi
if $A$ is sufficiently large, where $c(L) \to 0$ as $L\to \infty$. Hence, \eqref{JJJJ00} implies that $H_L(\varphi)>0$ if $\varphi \neq 0$, which shows that $\varphi=0$  is  the unique minimizer for every $L \ge 1$. Moreover, the estimate \eqref{JJJJ00} implies the quantitative stability \eqref{JJJJ0}.

\end{proof}

\begin{remark}\rm \label{REM:CRIT01}

A direct application of the GNS inequality \eqref{GNS00} without Lemma \ref{LEM:var} does not characterize the critical value of $A$ given in \eqref{CRA0}.

If $A<| H_{0, 1}^{*}|$, then from the argument in \eqref{qq}, we have
$\inf_{\phi \in H^1(\R^2)}H(\phi)=-\infty$. In other words, it drives the ground state energy towards $-\infty$.  Hence, one does not expect the construction of the $\Phi^3_2$-measure if $A<| H_{0, 1}^{*}|$ to be possible, even on the finite volume $\T^2_L$. It would be interesting to see whether the $L$-periodic $\Phi^3_2$-measure can be constructed as a probability measure in the full range $A \ge | H_{0, 1}^{*}|$, especially the critical case $A =| H_{0, 1}^{*}|$.

\end{remark}

\section{Ultraviolet stability for $\Phi^3_2$-measure}\label{SEC:ULTS}

In this section, we first address the small-scale (ultraviolet) singularities and give a variational characterization of the free energy $\log Z_L$.

\subsection{Bou\'e-Dupuis variational formalism for the Gibbs measure}
\label{SUBSEC:var}

In this subsection, we introduce the main framework to analyze expectations of certain random fields under the Gaussian measure $\mu_L$.

Let $(\Omega, \mathcal{F},\mathbb{P})$ be a probability space on which is defined a space-time white noise $\xi_L$ on $\T^2_L \times \R_{+}$. Let $W_L(t)$ be the cylindrical Wiener process on $L^2(\T_L^2)$ with respect to the underlying probability measure $\PP$. That is,
\begin{align*}
W_L(t) =\sum _{\ld \in \Z^2_L } B_\ld(t)e^L_\ld
\end{align*}

\noi
where $\{ B_\ld \}_{\ld \in \Z^2_L}$ is defined by  $B_\ld(t)=\jb{\xi_L, \ind_{[0,t]} \cdot e^L_\ld  }_{\T^2_L\times \R} $. Here, $\jb{\cdot, \cdot}_{\T^2_L \times \R}$ denotes the duality pairing on $\T^2_L \times \R$ and $\xi_L$ is a space-time white noise on $\T^2_L \times \R_{+}$. Then, we see that $\{B_\ld \}_{\ld \in \Z^2_L}$ is a family of mutually independent complex-valued\footnote{In particular, $B_0$ is  a standard real-valued Brownian motion.} Brownian motions conditioned that $B_{-\ld}=\cj{B_\ld}$, $\ld \in \Z^2_L$. We  then define a centered Gaussian process $ \<1>_L(t) $ by 
\begin{align}
\<1>_L(t)  &= \jb{\nabla}^{-1} W_L(t).
\label{Rol}
\end{align}

\noi
Then, we have $\Law ( {\<1>}_L(1)) = \mu_L$.  By setting  $\<1>_{L,N}(t) = \P_N {\<1>}_L(t)$, we have   $\Law ({\<1>}_{L,N}(1)) = (\P_N)_\#\mu_L$. We define the second and third Wick powers of $\<1>_{L,N}$ as follows 
\begin{align}
\<2>_{L,N}(t)
&=\<1>_{L,N}^2(t) - \<tadpole>_{L,N}(t), \label{cher18}\\
\<3>_{L,N}(t)
&=\<1>_{L,N}^3(t) -3\<tadpole>_{L,N}(t) \<1>_{L,N}(t) \label{cher19},
\end{align}

\noi
where a Riemann sum approximation gives
\begin{align*}
\<tadpole>_{L,N}(t):=\E \Big[|\<1>_{L,N}(t) |^2\Big] =\sum_{\substack{\ld \in \Z^2_L\\ |\ld|\le N}} \frac{1}{\jb{\ld}^2} \frac{1}{L^2}  \sim t \log N.
\end{align*}

\noi
The second and third Wick powers of $\<1>_{L,N}(t)$ are the space-stationary stochastic processes. In particular, $\<2>_{L,N}(1)$ and $\<3>_{L,N}(1)$ are equal in law to $:\! \phi_N^2 \!: \,$ and $:\! \phi_N^3  \!: \, $ in \eqref{Wickqu} and \eqref{Wickcu}, respectively.

Next, let $\mathbb{H}_a=\mathbb{H}_a(\T^2_L) $ denote the space of drifts, which are the progressively measurable processes\footnote{With respect to the filtration $\mathcal{F}_t=\sigma(B_\lambda(s), \lambda\in \mathbb{Z}_L^2, 0\le s\le t)$.} belonging to $L^2([0,1]; L^2(\T^2_L))$, $\PP$-almost surely. We are now ready to state  the  Bou\'e-Dupuis variational formula \cite{BD, Ust};
in particular, see Theorem 7 in~\cite{Ust}. See also Theorem 2 in \cite{BG}.

\begin{lemma}\label{LEM:var3}
Let $\<1>_L(t)=\jb{\nabla}^{-1} W_L(t)$ be as in \eqref{Rol}.
Fix $N \in \N$.
Suppose that  $F:C^\infty(\T^2_L) \to \R$
is measurable such that $\E\big[|F(\P_N \<1>_L(1))|^p\big] < \infty$
and $\E\big[|e^{-F(\P_N \<1>_L(1))}|^q \big] < \infty$ for some $1 < p, q < \infty$ with $\frac 1p + \frac 1q = 1$.
Then, we have
\begin{align*}
- \log \E\Big[e^{-F(\P_N \<1>_L(1))}\Big]
= \inf_{\dr_L \in \mathbb{H}_{a}(\T^2_L) }
\E\bigg[ F(\P_N \<1>_L(1) + \P_N \Dr_L(1)) + \frac{1}{2} \int_0^1 \| \dr_L(t) \|_{L^2(\T^2_L) }^2 dt \bigg], 
\end{align*}

\noi
where  $\Dr_L$ is  defined by 
\begin{align}
 \Dr_L(t) = \int_0^t \jb{\nabla}^{-1} \dr_L(t') dt'
\label{DR}
\end{align}

\noi
and the expectation $\E = \E_\PP$
is an 
expectation with respect to the underlying probability measure~$\PP$.

\end{lemma}

In the following, we set $\<1>_{L,N}=\P_N {\<1>}_L(1)$ and $\Dr_{L,N}=\P_N \Dr_L(1)$ for $N\in \N\cup \{\infty\} $ and finite $L>0$.

\subsection{Ultraviolet stability of Wick powers}
 We present a lemma on  pathwise regularity estimates  of $\<1>_{L,N}(t), \<2>_{L,N}(t), \<3>_{L,N}(t)$, and $\Dr_L(t)$. In particular, we also specify the growth rate as $L\to \infty$ for the stochastic objects.

\begin{lemma}  \label{LEM:Cor00}

\textup{(i)} 
For any finite $p \ge 2$, $1\le r\le \infty$, $t\in [0,1]$, and $\eta>0$, each Wick power in \eqref{cher18} and \eqref{cher19} converges to a limit  in $L^p(\O; W^{-\eta,r}(\T^2_L))$ as $N\to \infty$  and almost surely in $ W^{-\eta,r}(\T^2_L)$.
Moreover, we have 
\begin{equation}
\begin{split}
\E \Big[& \|\<1>_{L,N}(t)\|_{W^{-\eta,r}(\T^2_L)}^p+\|\<2>_{L,N}(t)\|_{W^{-\eta,r}(\T^2_L)}^p+\|\<3>_{L,N}(t)\|_{W^{-\eta, r}(\T^2_L)}^p
\Big] \les    L^2   <\infty, 
\label{SO3}
\end{split}
\end{equation}

\noi
uniformly in\footnote{When $N=\infty$, the statement concerns the norms of the limiting objects.} $N\in \N \cup \{\infty\}$ and $t \in [0, 1]$.

\smallskip

\noi
\textup{(ii)} 
For any $N \in \N \cup \{\infty\}$, we have 
\begin{align}
\E \bigg[ \int_{\T_L^2} \<3>_{L,N}(1) dx \bigg]
&= 0 \label{Aj0}\\
\E \Bigg[ \bigg|\int_{\T_L^2} \<3>_{L,N}(1) dx\bigg|^2 \Bigg] &\sim L^2 \label{Aj1} \\
\E \Bigg[ \bigg|\int_{\T_L^2} \<2>_{L,N}(1) dx\bigg|^2 \Bigg] &\sim L^2
\end{align}

\noi
as $L\to \infty$, where the implicit constant is uniform in $N \ge 1$.

\smallskip

\noi
\textup{(iii)}
The drift term $ \dr_L\in \mathbb{H}_a(\T^2_L) $ has the regularity
of the Cameron-Martin space, that is, for any $\dr_L \in \mathbb{H}_{a}(\T^2_L) $, we have
\begin{align}
\| \Dr_L(1) \|_{ H^{1}(\T^2_L) }^2 \leq \int_0^1 \|  \dr_L(t) \|_{ L^2(\T^2_L)}^2dt.
\label{CSH1}
\end{align}
\end{lemma}

\begin{proof}
For the proof, see \cite{OSeoT}, with $\T^2$ replaced by $\T^2_L$,  depending on  $L$.
\end{proof}

\begin{remark}\rm \label{REM:size}
Regarding the interaction potential $\textbf{V}^L(\phi)$, which is the limit of $\textbf{V}^L_N$ as defined in \eqref{Den1}, we write
\begin{align*}
\textbf{V}^L(\phi):=\textbf{V}^{(1),L}(\phi)+\textbf{V}^{(2),L}(\phi)
\end{align*}

\noi 
where
\begin{align*}
\textbf{V}^{(1),L}(\phi)&=\frac \s3  \int_{\T^2_L}   :\! \phi^3 \!: \, dx\\
\textbf{V}^{(2),L}(\phi)&=A\bigg(\int_{\T^2_L}  :\! \phi^2 \!: \, dx   \bigg)^2.
\end{align*} 

\noi
Thanks to Lemma \ref{LEM:Cor00} (ii), we have   $\E_{\mu_L}\Big[ \big( \textbf{V}^{(1),L}(\phi)\big)^2 \Big] \sim L^2$ and $\E_{\mu_L}\Big[ \textbf{V}^{(2),L}(\phi) \Big] \sim L^2$.  Therefore, the potential energy $\textbf{V}^{(1),L}(\phi)$ grows linearly $L$ as $L \to \infty$, while $\textbf{V}^{(2),L}(\phi)$ behaves quadratically $L^2$ as $L \to \infty$.
Hence, we conclude that $\textbf{V}^L(\phi)$ grows like $L^2$.







\end{remark}

\subsection{Gamma convergence}
In this subsection, we study the $\G$-convergence (Proposition \ref{PROP:Gamma}) of the variational problem by taking the ultraviolet limit $N\to \infty$, following an idea in \cite{BG}. This allows us to remove the ultraviolet cutoff $\P_N$ when applying Lemma \ref{LEM:var3}, and obtain a variational characterization for $Z_L$.

By the Bou\'e-Dupuis formula (Lemma \ref{LEM:var3}), the partition function $Z_{L,N}$ with ultraviolet $\P_N$ and infrared cutoffs $\T^2_L$, defined by
\begin{align}
Z_{L,N}=\int e^{-\frac{\s}{3} \int_{\T^2_L} :\, \phi_N^3 :\, dx  -A\big( \int_{\T^2_L} :\, \phi_N^2: \, dx   \big)^2 }  d \mu_L(u),
\end{align}

\noi
has the variational expression
\begin{align}
-\log Z_{L,N}&=\inf_{\Dr \in \Ha^1(\T^2_L) }\E\bigg[\textbf{V}^L_N(\<1>_L+\Dr_{L})+\frac 12 \int_0^1 \|\dot{\Dr}_L(t) \|^2_{H^1(\T^2_L)} dt\bigg] \notag \\
&=\inf_{\Dr \in \Ha^1(\T^2_L) }\E\bigg[\Phi_{N, L}(\Xi_L, \Dr_L)+A \bigg( \int_{\T^2_L}\Dr_{L,N}^2 dx\bigg)^2 +\frac 12 \int_0^1 \|\dot{\Dr}_L(t) \|^2_{H^1(\T^2_L)} dt \bigg]
\label{var1}
\end{align}

\noi
where $\Xi_L=(\<1>_L, \<2>_L, \<3>_L)$ and $\Phi_{L,N}=\Phi^{(1)}_{L,N}+\Phi^{(2)}_{L,N}$
\begin{align}
\Phi_{L,N}^{(1)}(\Xi_L, \Dr_L)&=
\frac {\s}3\int_{\T^2_L} \<3>_{L,N} dx+\s \int_{\T^2_L} \<2>_{L,N} \Dr_{L,N} dx+\s \int_{\T^2_L} \<1>_{L,N} \Dr_{L,N}^2 dx+\frac{\s}{3}\int_{\T^2_L}  \Dr_{L,N}^3 dx \notag\\
\Phi_{L,N}^{(2)}(\Xi_L, \Dr_L)&=A\bigg\{\int_{\T^2_{L}} (\<2>_{L,N} +2 \<1>_{L,N} \Dr_{L,N}+\Dr_{L,N}^2) dx   \bigg\}^2-A \bigg( \int_{\T^2_L}\Dr_{L,N}^2 dx\bigg)^2.
\label{Phi1}
\end{align}

\noi
The positive terms $A \|\Dr_{L,N} \|_{L^2(\T^2_L)}^4$ and $\frac{1}{2}\int_0^1 \| \dot{ \Dr}_L(t) \|_{H^1}^2 dt$ in \eqref{var1} ensure that the free energy $\log Z_{L,N}$ is finite uniformly in $N$ for each fixed $L>0$. For convenience of notation, we set
\begin{align*}
\int_0^1 \| \dot{ \Dr}_L(t) \|_{H^1}^2 dt:=\|\Dr_L \|_{\mathbb{H}^1}^2.
\end{align*}

\noi
We now study the $\G$-convergence of the variational problem in \eqref{var1} as the ultraviolet cutoff $\P_N$ is removed (i.e.~as $N\to \infty$).

\begin{definition}
Let $(X,\mathcal{T})$ be a topological space and $\{F_n\}_{n\in \overline{\N}}$ \footnote{$\overline{\N}$ means the set of extended natural numbers, i.e.~$\N\cup\{\infty\}$ }be a sequence of functionals on $X$. The sequence of functionals $\{F_n\}_{n\in \N}$ $\G$-converges to the $\G$-limit $F_\infty$ if the following two conditions hold:
\begin{itemize}
\item[(i)] For every sequence $x_n \to x $ in $X$, we have 
\begin{align*}
F_\infty(x)\le \liminf_{n\to \infty} F_n(x_n).
\end{align*}

\smallskip

\item[(ii)] For every point $x\in X$, there exists a sequence $\{x_n\}$ (recovery sequence) converging to $x$ in $X$ such that we have 
\begin{align*}
\limsup_{n\to \infty} F_n(x_n)\le F_\infty(x).
\end{align*}

\end{itemize}

\end{definition}

We also need the notion of \emph{equicoercivity}.

\begin{definition}
A sequence of functionals denoted as $\{F_n\}_{n\in \N}$ is said to be equicoercive if there is a compact set $K\subset X$ such that, for every $n\in \N$, the following condition holds:
\begin{align*}
\inf_{x\in K} F_n(x)=\inf_{x\in X} F_n(x).
\end{align*} 

\end{definition}

One important implication of $\G$-convergence and equicoercivity is the convergence of the minima.

\begin{proposition}
Suppose that $\{F_n\}_{n\in \N}$  $\G$-converges to $F_\infty$ and $\{F_n\}_{n\in \N}$ is equicoercive. Then, $F_\infty$ possesses a minimum. Moreover, we have the convergence of minima
\begin{align*}
\min_{x\in X} F_\infty(x)=\lim_{n\to \infty} \inf_{x\in X} F_n(x).
\end{align*}

\end{proposition}

Our goal in this section is to establish the $\G$-convergence of the variational problem in \eqref{var1} as $N\to \infty$ (Proposition \ref{PROP:Gamma}). For this, we relax the variational problem presented in \eqref{var1}. Instead of solving the problem over $\Ha^1$ with the strong topology, we consider a problem on the space of probability measures with a weak topology. Define
\begin{align}
\mathcal{X}_L:=\Big\{ \mu=\Law(\Xi_L, \Dr_L) \in \mathcal{P}(\vec{W}^{-\eta,r} \times \mathbb{H}^1_w):   \Dr_L \in \Ha^1 \quad \text{and} \quad \E_{\mu_L}\Big[ \| \Dr \|^2_{\mathbb{H}^1} \Big] <\infty   \Big\},
\label{Xnot}
\end{align}

\noi
where $\vec{W}^{-\eta,r}=W^{-\eta,r}\times W^{-\eta,r} \times W^{-\eta,r}$ for any fixed $1\le r <\infty $ and $\mathcal{P}(\vec{W}^{-\eta,r} \times \mathbb{H}^1_w)$ is the space of Borel probability measures on $\vec{W}^{-\eta,r}\times \mathbb{H}^1_w$. Here $\mathbb{H}^1_w$ means that  $\mathbb{H}^1$ is equipped with the weak topology. We will set up a minimization problem over the space $\mathcal{X}_L$ of distributions $\mu_L=\Law_{\PP}(\Xi_L, \Dr_L)$, where $\Xi_L=(\<1>_L, \<2>_L, \<3>_L) $ is fixed, and $\Dr_L$ varies within $\Ha^1$, employing the weak topology.

We now complete the space $\mathcal{X}_L$:
\begin{align*}
\cj {\mathcal{X}}_L:=\Big\{ \mu \in \mathcal{P}(\vec{W}^{-\eta,r} \times \mathbb{H}^1_w):&\;
\mu_{n} \to \mu \;  \text{weakly for some $\{\mu_{n}\}_{n\in \N} \in \mathcal{X}_L$}\\
& \text{and} \; \sup_{n \in \N}\E_{\mu_{n}}\Big[ \| \Dr \|^2_{\mathbb{H}^1} \Big] <\infty   \Big\}.
\end{align*}

\noi
Thus $\mathcal{X}_L$ is equipped with the following topology: $\{\mu_{n}\}_{n \in \N}$ in $\cj{ \mathcal{X}}_L$ converges to $\mu$ if (i) $\mu_n$ converges to $\mu$ weakly on $\vec{W}^{-\eta,r} \times H^1_w$  and (ii) $\sup_{n \in \N}\E_{\mu_n}\big[ \| \Dr \|^2_{\mathbb{H}^1} \big]<\infty$. Each element $\mathcal{X}_L$ has first marginal equal to to $\Law_{\PP}(\Xi_L)$, and this fact extends to $\cj{\mathcal{X}}$. Passing to this space ensures compactness.

To present the relaxation of the variational problem, define, for $N \in \N \cup \{\infty\}$,
\begin{align}
F^{L}_N(\Dr_L)&=\E_{\PP}\bigg[\textbf{V}^L_N(\<1>_{L}+\Dr_{L})+\frac 12 \int_0^1 \|\dot{\Dr}_{L}(t) \|^2_{H^1(\T^2_L)} dt\bigg] \notag  \\
&=\E\bigg[\Phi_{L,N}(\Xi_L, \Dr_L)+A \bigg( \int_{\T^2_L}\Dr_{L,N}^2 dx\bigg)^2 +\frac 12 \int_0^1 \|\dot{\Dr}_{L}(t) \|^2_{H^1(\T^2_L)} dt \bigg]
\label{var30}
\end{align}

\noi
where $\Phi_{L, N}=\Phi^{(1)}_{L, N}+\Phi^{(2)}_{L, N}$ is given in \eqref{phi1}. When $N=\infty$, the projection is interpreted the identity operator (i.e.~$\P_N=\Id$).  We substitute the initial variational problem \eqref{var1} with a new variational problem over $\mathcal{X}_L$ as follows
\begin{align}
\inf_{\Dr \in \Ha^1} F_N^L(\Dr)= \inf_{\mu \in \mathcal{X}_L} F_N^L(\mu).
\label{RELAX}
\end{align}

\noi
Here $\E_\mu$ denotes the expectation with respect to the probability measure $\mu$.  The following lemma shows that the variational problem on $\mathcal{X}_L$ and $\cj {\mathcal{X}}_L$ are equivalent. In particular, the infimum is achieved within $\cj{\mathcal{X}}$. For the proof of Lemma \ref{LEM:relax}, see \cite[Lemma 15, 18]{BG} or \cite[Lemma 8]{BG1}.


\begin{lemma}\label{LEM:relax}
Let $L> 0$ and $N \in \N \cup \{\infty\}$. Then, we have
\begin{align*}
\inf_{\mu \in \mathcal{X}_L} F_N^L(\mu)=\min_{ \mu \in \cj{\mathcal{X}}_L}F_N^L(\mu).
\end{align*}

\noi
Here the infimum is attained at an element in $\cj{\mathcal{X}_L}$.
\end{lemma}

The following lemma establishes compactness on $\cj {\mathcal{X}_L}$. For the proof, see \ref{LEM:comp}, see \cite[Lemma 10]{BG}. 

\begin{lemma}\label{LEM:comp}
Let $L>0$ and $\mathcal{K}$ be a subset of $\cj {\mathcal{X}_L}$ such that $\sup_{\mu \in \mathcal{K} } \E_{\mu}\Big[\|\Dr_L \|_{\mathbb{H}^1}^2\Big]<\infty$. Then, $\mathcal{K}$ is compact in $\cj {\mathcal{X}_L}$.  
\end{lemma}


We are now ready to prove the following proposition that allows us to obtain the variational characterization of the grand-canonical partition function $Z_L$ without the ultraviolet cutoff $\P_N$. 

\begin{proposition}[Gamma convergence]\label{PROP:Gamma}
Let $L>0$. Then, the sequence of functional $\{F^L_N\}_{N \in \N}$ $\G$-converges to $F_\infty^L$ on $\cj{\mathcal{X}_L}$ as $N \to \infty$. Moreover, we have
\begin{align}
\lim_{N\to \infty} \inf_{\Dr_L \in \Ha^1} F^{L}_N(\Dr_L)=\inf_{\Dr_L \in \Ha^1 }F_\infty^{L}(\Dr_L)
\label{Gconvm}
\end{align}

\noi
where the functionals $F^L_N$ and $F_\infty^L$ are given as in \eqref{var30}. In particular, the grand-canonical partition function $Z_L$ in \eqref{PART0} is given by
\begin{align}
-\log Z_L&=\inf_{\Dr_L \in \Ha^1 }F_{\infty}^L(\Dr_L)
\end{align}

\noi
for every $L>0$.

\end{proposition}


\begin{proof}
Thanks to the relaxed variational problems coming from \eqref{RELAX} and Lemma \ref{LEM:relax}, it suffices to consider the variational problem \eqref{Gconvm} over $\cj {\mathcal{X}}_L$. We first prove the following liminf inequality  
\begin{align}
F^{L}_\infty(\mu)\le \liminf_{N\to\infty}F^L_N(\mu_N)
\label{liminf0}
\end{align}

\noi
when $\mu_N \to \mu $ in $\bar{\mathcal{X}}_L$. We may assume that $\sup_{N}F^N_L(\mu_N)<\infty$. Otherwise, there is nothing to prove.
By exploiting the Skorokhod's representation theorem, there exists random variables $\{X_N, \zeta_N\}_{N\in \N}$ and $\{X_\infty, \zeta_\infty\}$ on a common probability space $(\wt \O, \wt \F, \wt \PP)$, with values in $\vec{W}^{-\eta,r} \times \mathbb{H}^1_w$ such that
\begin{align}
\Law_\PP(X_N,\zeta_N)=\mu_N \qquad \text{and} \qquad \Law_\PP(X_\infty,\zeta_\infty)=\mu 
\label{ZX1}
\end{align}

\noi
for every $N \ge 1$. Furthermore, we have the following almost sure convergence 
\begin{align}
X_N \to X  \quad \text{in} \quad \vec{W}^{-\eta,r}  \\
\zeta_N \to  \zeta  \quad \text{in} \quad \mathbb{H}^1_w 
\end{align}  

\noi
as $N \to \infty$. It can be easily proven that for any sequence $\{X_N,\zeta_N\}$ satisfying 
$X_N \to X_{\infty}$ in $\vec{W}^{-\eta,r}$ and $\zeta_N \to \zeta_{\infty}$ in $\mathbb{H}^1_w$, we have
\begin{align}
\lim_{N\to \infty} \Phi_{L,N}(X_N, \zeta_N)=\Phi_{L, \infty}(X_\infty, \zeta_\infty).
\label{ZX3}
\end{align}

\noi
Thanks to the pathwise regularity estimates in Lemma \ref{LEM:Dr1}, we have the following pathwise bound on the same probability space 
\begin{align}
\Phi_{L,N}(X_N,\zeta_N)+A\|\zeta_N \|_{L^2}^4+\frac{1}{2}\|\zeta_N \|_{H^1}^2+H(X_N) \ge 0
\label{ZX2}
\end{align}

\noi
for some random variable $H(X_N) \in L^1(d\wt \PP)$, uniformly in $N$, such that $\E_{\wt \PP}\big[H(X_N) \big]=\E_\PP\big[ H(\Xi_N) \big]$ for every $N$, where $\Xi_N=\big(\<1>_{L,N}, \<2>_{L,N}, \<3>_{L,N}\big)$. For example, we can choose $H(X_N)=C(1+\| X_N\|_{\vec{W}^{-\eta,r}}^p )$ for some large $C\gg 1$ and $p\gg 1$.  It follows from \eqref{ZX1}, \eqref{ZX2}, \eqref{ZX3}, and Fatou's lemma that
\begin{align*}
\liminf_{N\to \infty} F^L_N(\mu_N)&=\liminf_{N \to \infty} \E_{\wt \PP} \bigg[ \Phi_{L,N}(X_N,\zeta_N) +A\| \zeta_N\|_{L^2}^4+\frac{1}{2} \|\zeta_N \|_{H^1}^2 \bigg]\\
&=\liminf_{N \to \infty} \Bigg\{ \E_{\wt \PP} \bigg[ \Phi_{L, N}(X_N,\zeta_N) +A\| \zeta_N\|_{L^2}^4+\frac{1}{2} \|\zeta_N \|_{H^1}^2 +H(X_N) \bigg] -\E\big[H(X_N)\big] \Bigg\}\\
&\ge \E_{\wt \PP} \liminf_{N\to \infty} \bigg[\Phi_{L,N }(X_N,\zeta_N) +A\| \zeta_N\|_{L^2}^4+\frac{1}{2} \|\zeta_N \|_{H^1}^2 +H(X_N)\bigg]- \E_{\PP}\big[H(\Xi)\big]\\
&= \E_{\wt \PP} \bigg[\Phi_{L,\infty}(X_\infty,\zeta_\infty) +A\| \zeta_\infty\|_{L^2}^4+\frac{1}{2} \|\zeta_\infty \|_{H^1}^2 \bigg]\\
&=F_\infty^L(\mu),
\end{align*}

\noi
from which we obtain \eqref{liminf0}.

Next, we prove that for every $\mu \in \bar{\mathcal{X}_L}$, there exists a sequence $\{\mu_N\}$ such that $\{\mu_N\}$ converges  to $\mu$ in $\bar{\mathcal{X}}_L$ and   
\begin{align}
\limsup_{N\to\infty}F^L_N(\mu_N)\le F^{L}_\infty(\mu). 
\label{limsup0}
\end{align}

\noi
Let $\mu \in \bar{\mathcal{X}}_L$. By setting $\mu_N:=\mu$ for every $N \ge 1$, we obtain $\mu_N \to \mu$ in $\bar{\mathcal{X}}_L$. We may assume that $F_\infty^L(\mu)<\infty$.  
Thanks to Lemma \ref{LEM:Cor00} and \ref{LEM:Dr1}, we have
\begin{align}
F_\infty^L(\mu)&\ge -cL^2+(1-\dl)
\E_{\mu} \bigg[  A\|\Dr_L \|_{L^2(\T^2_L)}^4+\frac 12 \|\Dr_L \|_{H^1(\T^2_L)}^2  \bigg] 
\label{AS1}
\end{align}

\noi
for some small $0<\dl \ll 1$ and $c>0$, where $L^2$ follows from computing the expected values of the higher moments for each component of $\Xi_{L,N}=\big(\<1>_{L,N}, \<2>_{L,N}, \<3>_{L,N}\big)$ in $W^{-\eta,r}$, uniformly in $N\ge 1$. From the assumption $F_\infty^L(\mu)<\infty$ and \eqref{AS1}, we have
\begin{align}
\E_{\mu} \bigg[  A\|\Dr_L \|_{L^2(\T^2_L)}^4+\frac 12 \|\Dr_L \|_{H^1(\T^2_L)}^2  \bigg]<\infty
\label{AS2}  
\end{align} 

\noi
for each fixed $L>0$. Then, by the definition of $F^L_N(\mu)$ in \eqref{var30}, Lemma \ref{LEM:Cor00}, \ref{LEM:Dr1}, and \eqref{AS2}, we can use the dominated convergence theorem to obtain
\begin{align*}
\lim_{N\to \infty} F^{L}_N(\mu_N)&=\lim_{N\to \infty} F^{L}_N(\mu)\\
&=\lim_{N\to \infty } \E_\mu\bigg[\Phi_{L,N}( \P_N \Xi_L, \P_N \Dr_L)+ A\| \P_N\Dr_{L} \|_{L^2}^4+\frac 12\| \P_N\Dr_{L} \|_{H^1}^2 \bigg]\\
&=F^{\infty}_L(\mu).
\end{align*}   

\noi
Hence, we obtain the result \eqref{limsup0}.

Finally, we show that  $\{F^L_N\}_{N \in \N}$ is equicoercive on $\cj {\mathcal{X}}_L$. Define 
\begin{align*}
\mathcal{K}:=\bigg\{ \mu \in \bar{\mathcal{X}}_L: \E_{\mu}\Big[   \| \Dr_L \|_{L^2}^4 \Big]+ \E_{\mu}\Big[\| \Dr_L \|_{\mathbb{H}^1}^2 \Big] \le K         \bigg\}
\end{align*}

\noi
for some sufficiently large $K \gg 1$, which will be specified below. 
Thanks to Lemma \ref{LEM:comp}, $\mathcal{K}$ is compact. By using Lemma \ref{LEM:Dr1} and \ref{LEM:Cor00}, we have
\begin{align}
\inf_{\mu \not \in \mathcal{K}} F_N^L(\mu)&\ge -c_1L^2+(1-\dl)
\inf_{\mu \not \in \mathcal{K}}\E\bigg[  A\|\Dr_L \|_{L^2(\T^2_L)}^4+\frac 12 \|\Dr_L \|_{H^1(\T^2_L)}^2  \bigg] \notag \\
&\ge -c_1L^2+c_2(1-\dl)K
\label{ZK}
\end{align}

\noi
for some $c_1, c_2>0$ and small $\dl>0$, where $L^2$ arises by computing the expected values of the higher moments for each component of $\Xi_{L,N}=\big(\<1>_{L,N}, \<2>_{L,N}, \<3>_{L,N}\big)$ in $W^{-\eta, r}$, uniformly in $N\ge 1$. Thanks to Lemma \ref{LEM:Dr1} and \ref{LEM:Cor00}, we have 
\begin{align}
\sup_N\inf_{\mu \in \bar{\mathcal X}_L} F_N^L(\mu)\le c_1L^2+ (1+\dl)\inf_{\mu  \in \bar{\mathcal{X}}_L}\E_{\mu}\bigg[  A\|\Dr_L \|_{L^2(\T^2_L)}^4+\frac 12 \|\Dr_L \|_{H^1(\T^2_L)}^2  \bigg]<\infty,
\label{ZK1}
\end{align}

\noi
Hence, it follows from \eqref{ZK}, \eqref{ZK1}, and choosing $K \gg 1$ sufficiently large that
\begin{align*}
\inf_{\mu \in \mathcal{K}} F_N^L(\mu)=\inf_{\mu \in \bar{\mathcal{X}}_L} F_N^L(\mu),
\end{align*} 

\noi
for every $N \ge 1$, from which we conclude that $\{F^L_N\}_{N\in \N}$ is equicoercive.

\end{proof}

We close this subsection by showing convergence of the Hamiltonian $H_L$ as the size of the torus goes to infinity (i.e.~$L\to \infty$).

\begin{lemma}\label{LEM:GAML}
There exists a large constant $A_0 \ge 1$ independent of $L$ such that for all $\s \in \R \setminus \{0\} $ and $A \ge A_0$,
\begin{align*}
\lim_{L\to \infty} \inf_{\phi \in H^1(\T^2_L)} H_L(\phi)=\inf_{\phi\in H^1(\R^2)} H(\phi)
\end{align*}

\noi
where 
\begin{align*}
H_L(\phi)=\frac 12\int_{\T^2_L} |\nb \phi|^2 dx+\frac{\s}{3} \int_{\T^2_L} \phi^3 dx +A\bigg(\int_{\T^2_L} \phi^2 dx \bigg)^2. 
\end{align*}

\end{lemma}

\begin{proof}

We first prove 
\begin{align}
\liminf_{L \to \infty} \inf_{\phi \in H^1(\T^2_L) }H_L(\phi) \ge \inf_{\phi\in H^1(\R^2)}H(\phi).
\label{JJJC0}
\end{align}

\noi
Thanks to the GNS inequality \eqref{GNS} on $\T^2_L$ (Lemma \ref{LEM:GNST}) and Young's inequality, we have 
\begin{align*}
H_L(\phi)\ge \frac {1-\dl}2 \int_{\T^2_L} |\nb \phi|^2 dx +(A-c(\dl)-c(L)) \bigg( \int_{\T^2_L} \phi^2 dx \bigg)^2 \ge 0
\end{align*}

\noi
if $A$ is sufficiently large, where $c(L) \to 0$ as $L\to \infty$, which implies 
\begin{align*}
\liminf_{L \to \infty} \inf_{\phi \in H^1(\T^2_L) }H_L(\phi) \ge 0.
\end{align*}

\noi
From Lemma \ref{LEM:MIN1}, we have $\inf_{\phi\in H^1(\R^2)} H(\phi)=0$ and so obtain the result \eqref{JJJC0}.

It remains to prove 
\begin{align}
\inf_{\phi\in H^1(\R^2)}H(\phi) \ge \limsup_{L \to \infty} \inf_{\phi \in H^1(\T^2_L) }H_L(\phi).
\label{JJJC1}
\end{align}

\noi
Let $u^{*}$ be a minimizer, namely, $\inf_{\phi\in H^1(\R^2)}H(\phi)=H(u^{*})$. Let $\{\varphi_L\}_{L\ge 1}$ be a sequence of smooth cutoff functions where $\varphi_L$ is supported on $\big[-\frac L8, \frac L8\big]^2$ and $\varphi_L=1$ on $\big[-\frac L{16}, \frac L{16}\big]^2$. Then, $\varphi_L u^{*}\in H^1(\T^2_L)$ and so $\{\varphi_L u^{*}\}_{L\ge 1}$ is a minimizing sequence. Hence, we obtain
\begin{align*}
\inf_{\phi\in H^1(\R^2)}H(\phi)=H(u^*)=\lim_{L\to \infty} H(\varphi_L u^*)&= \lim_{L\to \infty} H_L(\varphi_L u^{*})\\
&\ge \limsup_{L \to \infty} \inf_{\phi \in H^1(\T^2_L) }H_L(\phi).
\end{align*}  

\noi
By combining \eqref{JJJC0} and \eqref{JJJC1}, we obtain the result.
 
\end{proof}

\section{Analysis of the free energy}\label{SEC:Free}

In this section, we analyze the behavior of the free energy $\log Z_L$ as $L \to \infty$. Our main goal is to establish the following large deviation estimate.

\begin{proposition}\label{PROP:free}
There exists a large constant $A_0 \ge 1$ independent of $L \ge 1$ such that for all $\s \in \R \setminus \{0\} $ and $A \ge A_0$, the grand-canonical partition function $Z_L$ satisfies
\begin{align*}
\lim_{L\to \infty} \frac{\log Z_{L}}{L^4}=-\inf_{\phi\in H^1 (\R^2)} H(\phi)
\end{align*}

\noi
where
\begin{align*}
H(\phi)&=\frac 12\int_{\R^2} |\nb \phi|^2 dx+\frac{\s}{3} \int_{\R^2} \phi^3 dx +A\bigg(\int_{\R^2} \phi^2 dx \bigg)^2.
\end{align*}

\end{proposition}

We prove Proposition \ref{PROP:free} by showing Lemma \ref{LEM:UP} and \ref{LEM:LB} in the following subsections.

\subsection{Upper bound for the free energy}

In this subsection, we investigate the limiting behavior of the free energy $\log Z_L$, concentrating on obtaining an upper bound.

\begin{lemma}\label{LEM:UP}
There exists a large constant $A_0 \ge 1$ independent of $L \ge 1$ such that, for all $\s \in \R \setminus \{0\} $ and $A \ge A_0$, we have 
\begin{align*}
\limsup_{L\to \infty} \frac{\log Z_L}{L^4}\le -\inf_{W\in H^1(\R^2)} H(W).
\end{align*}

\end{lemma}

\begin{proof}
Thanks to Proposition \ref{PROP:Gamma} and Lemma \ref{LEM:Cor00} (iii), the grand-canonical partition function can be expressed without the ultraviolet cutoff $\P_N$ as follows
\begin{align}
\log Z_L&=\sup_{\dr_L \in \Ha }\E \bigg[-\textbf{V}^L(\<1>_L+\Dr_L)-\frac 12 \int_0^1 \|\dr_L(t) \|^2_{H^1(\T^2_L)} \bigg] \notag \\
&\le \sup_{\Dr_L \in \mathcal{H}^1(\T^2_L) }\E \bigg[-\textbf{V}^L(\<1>_L+\Dr_L)-\frac 12 \|\Dr_L \|_{H^1(\T^2_L)}^2 \bigg] 
\label{Up11}
\end{align}

\noi
where $\mathcal{H}^1$ represents the collection of drifts $\Dr_L$, characterized as processes that belong to $H^1(\T^2)$ $\PP$-almost surely (possibly non-adapted). For any $\Dr_L\in \mathcal{H}_x^1(\T^2_L)$, we perform the change of variable $L^2W(L\cdot):=\<1>_{L,M}+\Dr_L$ where $\<1>_{L,M}=\P_M \<1>_L$. Set
\begin{align}
\Dr_L=-\<1>_{L,M}+W_L
\label{ch00}
\end{align}

\noi
where $W_L:=L^2W(L\cdot)$ for some $W\in \mathcal{H}^1(\T^2_{L^2})$. From \eqref{Up11} and \eqref{ch00}, we have 
\begin{align}
\log Z_{L} & \le  \sup_{ W \in \mathcal{H}^1(\T^2_{L^2}) } \E \bigg[-\textbf{V}^L\big( (\<1>_L-\<1>_{L,M} )+W_L   \big) -\frac 12 \| \<1>_{L,M} \|_{H^1(\T^2_L)}^2-\frac 12 \| W_L\|_{H^1(\T^2_{L})}^2 \notag \\
&\hphantom{XXXXXXXXX}-\int_{\T^2_L}\jb{\nb}\<1>_{L,M} \jb{\nb} W_L dx \bigg] \notag \\
& \le  \sup_{W \in \mathcal{H}^1(\T^2_{L^2}) }  \E \bigg[-\textbf{V}^L\big( (\<1>_L-\<1>_{L,M} )+W_L   \big)+\Big(c(\dl)-\frac 12\Big) \| \<1>_{L,M}\|_{H^1(\T^2_L)}^2-\frac {1-\dl}2 \| W_L\|_{H^1(\T^2_L)}^2 \bigg]
\label{Upp2}
\end{align}

\noi
where we used Young's inequality to find that for any $\dl>0$, 
\begin{align*}
\bigg| \int_{\T^2_L} \jb{\nb} \<1>_{L,M}  \jb{\nb}W_L dx \bigg| \le c(\dl) \| \<1>_{L,M} \|_{H^1(\T^2_L)}^2+\frac {\dl}2 \| W_L\|_{H^1(\T^2_L)}^2.
\end{align*}

\noi
With the change of variable given by \eqref{ch00}, we can express
\begin{align}
\textbf{V}^L(\<1>_L+\Dr_L)&=\int_{\T^2_L}  :\! \big( (\<1>_L-\<1>_{L,M})+W_L\big)^3 \!:  dx+A\bigg(\int_{\T^2_L} :\! \big( (\<1>_L-\<1>_{L,M})+W_L\big)^2 \!:     \bigg)^2 \notag \\
&=\int_{\T^2_L}  :\! \big(\wt {\<1>}_{L,M}+W_L\big)^3 \!:  dx+A\bigg(\int_{\T^2_L} :\! \big(\wt {\<1>}_{L,M}+W_L \big)^2 \!: dx    \bigg)^2.
\label{AS0}
\end{align}

\noi
Here, $\P_N \wt {\<1>}_{L,M}=\P_N(\<1>_L-\<1>_{L,M} )$ represents a new Gaussian field whose variance is given by
\begin{align*}
\E|\wt  \P_N {\<1>}_{L,M}(x)|^2=\sum _{\substack{\ld \in \Z^2_L \\ M<|\ld|\le N}} \frac{1}{\jb{\ld^2}}\frac{1}{L^2}=\sum_{\substack{n\in \Z^2\\ LM<|n|\le LN}}\frac{1}{\jb{\frac{n}{L}}^2}\frac{1}{L^2}\sim\int_{\R^2} \ind_{\{ M < |y|\le N  \}} \frac{dy}{1+|y|^2}
\end{align*}

\noi
for any $x\in \T^2_L$ as $L\to \infty$.  

For any Gaussian $X$ and $\s_1, \s_2$ in $\R$, we have 
\begin{align}
H_1(X;\s_1)&=H_1(X;\s_2) \notag \\
H_2(X;\s_1)&=H_2(X;\s_2)-(\s_1-\s_2) \notag \\
H_3(X;\s_1)&=H_3(X;\s_2)-3(\s_1-\s_2)H_1(X,\s_2)
\label{NWICKH}
\end{align}

\noi
where $H_k(x;\s)$ is the Hermite polynomial of degree $k$. Defining $\wt {\<2>}_{L,M}, \wt {\<3>}_{L,M}$ and the corresponding Wick powers relative to the Gaussian $\wt {\<1>}_{L,M}$, it follows from \eqref{NWICKH} that
\begin{align}
\int_{\T^2_L}  :\! \big(\wt {\<1>}_{L,M}+W_L\big)^3 \!:  dx&=\int_{\T^2_L}  :\! \big(\wt {\<1>}_{L,M}+W_L\big)^3 \!:_M  dx-3C_M\int_{\T^2_L} (\wt{\<1>}_{L,M}+W_L  ) dx   \notag \\
&=\int_{\T^2_L} \wt{\<3>}_{L,M} dx+3 \int_{\T^2_L} \wt{\<2>}_{L,M} W_{L} dx+3 \int_{\T^2_L} \wt{\<1>}_{L,M} W_{L}^2 dx+\int_{\T^2_L}  W_{L}^3 dx \notag \\
&\hphantom{X}-3C_M\int_{\T^2_L} (\wt{\<1>}_{L,M}+W_L  ) dx
\label{PAR1}
\end{align}

\noi
and
\begin{align}
\int_{\T^2_L} :\! \big(\wt {\<1>}_{L,M}+W_L \big)^2 \!: dx&=\int_{\T^2_L} :\! \big(\wt {\<1>}_{L,M}+W_L \big)^2 \!:_M dx-C_M \notag \\
&=\int_{\T^2_{L}} \wt{\<2>}_{L,M}dx +2 \int_{\T^2_L} \wt{\<1>}_{L,M} W_L dx+ \int_{\T^2_L} W_{L}^2 dx -C_M  
\label{PAR2}
\end{align}

\noi
where
\begin{align*}
C_M:&=\lim_{N \to \infty }\Big( \sum_{\substack{n\in \Z^2\\ |n|\le LN}}\frac{1}{\jb{\frac{n}{L}}^2}\frac{1}{L^2}-\sum_{\substack{n\in \Z^2\\ LM<|n|\le LN}}\frac{1}{\jb{\frac{n}{L}}^2}\frac{1}{L^2} \Big)\\
&=\sum_{\substack{n\in \Z^2\\ |n|\le LM}}\frac{1}{\jb{\frac{n}{L}}^2}\frac{1}{L^2}\sim\int_{\R^2} \ind_{\{ |y|\le M  \}} \frac{dy}{1+|y|^2} \sim \log M 
\end{align*}

\noi
as $M\to \infty$. Thus, from \eqref{PAR1}, \eqref{PAR2}, Lemma \ref{LEM:Dr1}, and Lemma \ref{LEM:Cor00}(i), it follows that for arbitrarily small $\delta > 0$
\begin{align}
\E\Bigg[ \bigg|\int_{\T^2_L}  :\! \big(\wt {\<1>}_{L,M}+W_L\big)^3 \!:  dx -\int_{\T^2_L} W_L^3 dx \bigg| \Bigg] \le \dl  \| W_L\|_{L^2(\T^2_L)}^4+\frac{\dl}{2} \| W_L \|_{H^1(\T^2_L)}^2  +O((\log M)^2L^2)
\label{Up20}
\end{align}

\noi
and
\begin{align}
A\E \Bigg[ \bigg| \int_{\T^2_L} :\! \big(\wt {\<1>}_{L,M}+W_L \big)^2 \!: dx+C_M \bigg|^2 \Bigg] \ge   \frac A2 \| W_L \|_{L^2(\T^2_L)}^4 - \frac 1{100} \| W_L \|_{H^1(\T^2_L)}^2-O(AL^2) -A(\log M)^2
\label{Up21}
\end{align}

\noi
where the term $O(L^2)$ comes from Lemma \ref{LEM:Cor00} (i) by computing the expectation of  the higher moments for each component of $(\wt {\<3>}_{L,M}, \wt {\<2>}_{L,M},  \wt {\<1>}_{L,M} \big)$ in $W^{-\eta,r}(\T^2_L)$ for $1 \le r \le \infty$. We also note that
\begin{align}
C_M\bigg| \int_{\T^2_L } W_L dx \bigg| \le  \| W_L \|_{L^2(\T^2_L)}C_ML\le  \dl \| W_L\|_{L^2(\T^2_L)}^4+O(\dl^{-1})+C_M^2L^2
\label{HAn1}
\end{align}

\noi
and
\begin{align}
\E\Big[\| \<1>_{L,M} \|_{H^1(\T^2_L)}^2\Big]=\sum_{\substack{\ld \in \Z^2_L\\ |\ld|\le M}} &=\sum_{\substack{n\in \Z^2 \\ |n|\le LM}}=O(L^2M^2). 
\label{LM2}
\end{align}

\noi
It follows from \eqref{Upp2}, \eqref{Up20}, \eqref{Up21}, \eqref{HAn1}, \eqref{LM2}, and undoing the scaling $W_L=L^2W(L\cdot)$ that 
\begin{align}
&\log Z_L  \notag \\
& \le \sup_{\Dr_L \in \mathcal{H}^1(\T^2_L) }  \E \bigg[-\textbf{V}^L\big( (\<1>_L-\<1>_{L,M} )+W_L   \big)+\Big(c(\dl)-\frac 12\Big) \| \<1>_{L,M}\|_{H^1(\T^2)}^2-\frac {1-\dl}2 \| W_L\|_{H^1(\T^2_L)}^2 \bigg] \notag \\
&\le \sup_{W \in \mathcal{H}^1(\T^2_{L^2}) }\E \bigg[ -\frac{\s}{3} \int_{\T^2_L} W_L^3 dx-\frac {A-\dl}2\|W_L \|_{L^2(\T^2_L)}^4 -\frac {1-2\dl}2 \| W_L\|_{H^1(\T^2_L)}^2  \notag \\
&\hphantom{XXXXXXXXXXXX}+\Big(c(\dl)-\frac 12\Big) \| \<1>_{L,M}\|_{H^1(\T^2_L)}^2 \Bigg]+O(L^2(\log M)^2)+O(\dl^{-1})\notag \\
&\le - L^4 \inf_{W\in H^1(\T^2_{L^2})}H_{L^2}^\dl(W) +O(L^2M^2)+O(\dl^{-1}) 
\label{U1}
\end{align}

\noi
where
\begin{align*}
H_{L^2}^\dl(W)=\frac{\s}{3}\int_{\T^2_{L^2}} W^3dx+\frac{A-\dl}{2}\bigg(\int_{\T^2_{L^2} }  W^2 dx \bigg)^2+\frac{1-2\dl}{2} \int_{\T^2_L} |\nb W|^2 dx.
\end{align*}

\noi
Therefore, by taking the limit first $L\to \infty$ in \eqref{U1} with Lemma \ref{LEM:GAML} and then $\dl \to 0$, we have
\begin{align*}
\limsup_{L\to \infty }\frac{\log Z_{L }}{L^4} \le - \inf_{W \in H^1(\R^2) } H(W),
\end{align*}

\noi
the desired result.

\end{proof}

\begin{remark}\rm \label{RM:ULT}
Following the arguments in the proof of Lemma \ref{LEM:UP} with the change of variable $\Dr_L=-\<1>_{L,N}+W_L$, we obtain
\begin{align}
\log Z_{L,N} \le - L^4 \inf_{W\in H^1(\T^2_{L^2})}H_{L^2}^\dl(W) +O(L^2N^2)+O(\dl^{-1}).
\label{bl0}
\end{align}

\noi
Here, the term $O(L^2N^2)$ arises from $\|\<1>_{L,N} \|_{H^1(\T^2_L)}^2$. Consequently, by taking the ultraviolet limit as $N\to \infty$, the truncated partition function $\log Z_{L,N}$ converges to $\log Z_L$. However, the right-hand side of \eqref{bl0} tends to infinity due to the term $O(L^2 N^2)$. Therefore, it is necessary to address the ultraviolet problem initially by using Proposition \ref{PROP:Gamma} and then separately control the infrared limit $L \to \infty$ as in the proof of Lemma \ref{LEM:UP}.
The same phenomena occur in the proofs of Lemma \ref{LEM:STAD} and \ref{LEM:LB}.

\end{remark}

The following lemma, whose proof follows similar lines to Lemma \ref{LEM:UP}, is used in the proof of Theorem \ref{THM:0}.

\begin{lemma}\label{LEM:STAD}
There exists a large constant $A_0 \ge 1$ and $c>0$ independent of $L \ge 1$ such that for any given $\eps>0$, all $\s \in \R \setminus \{0\} $ and $A \ge A_0$,
\begin{align*}
\limsup_{L\to \infty} \frac{\E_{\mu_L}\Big[  \exp\big\{-\textbf{V}^L(\phi) \big \} \ind_{ \{ \phi \not \in S_{L} \} }  \Big] }{L^4}\le    - c\eps^4
\end{align*}

\noi
where 
\begin{align*}
S_{L}=\big\{ \phi\in H^{-\eta}(\T^2_L): \| L^{-2}\phi(L^{-1}\cdot ) \|_{H^{-\eta}(\T^2_{L^2} )}   < \eps   \big\}.
\end{align*}

\end{lemma}

\begin{proof}

We first note that
\begin{align}
\log Z_L(S_{L}^c):&=\E_{\mu_L} \Big[  \exp\big\{-\textbf{V}^L(\phi) \big \} \ind_{ \{ \phi \not \in S_{L} \} } \Big]\notag \\
& \le \E_{\mu_L} \Big[  \exp\big\{-\textbf{V}^L(\phi) \ind_{ \{ \phi \not \in S_{L} \} }    \big \}  \Big]. 
\label{Up0}
\end{align}

\noi
We proceed as in the proof of Lemma \ref{LEM:UP} with considering $\ind_{ \{ \phi \not \in S_{L} \} }$. It follows from \eqref{Up0} and the analog of Proposition \ref{PROP:Gamma} with $\ind_{ \{ \phi \not \in S_{L} \} }$ that
\begin{align}
\log Z_L(S_{L}^c)&\le\sup_{\dr_L \in \Ha }\E \bigg[-\textbf{V}^L(\<1>_L+\Dr_L)\ind_{ \big\{(\<1>_L+\Dr_L) \notin S_{L} \big\}  }-\frac 12 \int_0^1 \|\dr_L(t) \|^2_{L^2(\T^2_L)} \bigg] \notag \\
&\le \sup_{\Dr_L \in \mathcal{H}^1(\T^2_L) }\E \bigg[-\textbf{V}^L(\<1>_L+\Dr_L)\ind_{ \big\{(\<1>_L+\Dr_L) \notin S_{L} \big\}  }-\frac 12 \|\Dr_L \|_{H^1}^2 \bigg] 
\label{Up1}
\end{align}

\noi
where the space $\mathcal{H}^1(\T^2_L)$ represents the set of $H^1(\T^2_L)$-valued random variables (these processes need not be adapted). For any $\Dr_L\in \mathcal{H}_x^1(\T^2_L)$, we perform the change of variable $L^2W(L\cdot):=\<1>_{L,M}+\Dr_L$ where $\<1>_{L,M}=\P_M \<1>_L$. Then, we write
\begin{align}
\Dr_L=-\<1>_{L,M}+W_L
\label{ch0}
\end{align}

\noi
where $W_L:=L^2W(L\cdot)$ for some $W \in \mathcal{H}^1(\T^2_{L^2})$. Define the set 
\begin{align}
\W_{L}=\Big\{W\in H^1(\T^2_{L^2}):  \|W \|_{H^1(\T^2_{L^2})} \ge \eps/2   \Big\}.
\label{SL}
\end{align}

\noi
If 
\[\<1>_L-\<1>_{L,M}+W_L \notin  S_{L},\]
then
\begin{align}
\|W \|_{H^1(\T^2_{L^2})} &\ge \|L^{-2} (\<1>_L-\<1>_{L,M})(L^{-1} \cdot)+W \|_{H^{-\eta}} - \|L^{-2}( \<1>_L-\<1>_{L,M})(L^{-1}\cdot) \|_{H^{-\eta}}\notag  \\
& \ge \eps -\|L^{-2}( \<1>_L-\<1>_{L,M})(L^{-1}\cdot) \|_{H^{-\eta}} \ge  
\eps/2,
\label{UP22}
\end{align}

\noi
by choosing sufficiently large $M=M(L) \gg 1$ with high probability, where we used the fact that 
\begin{align*}
\PP\Big\{  \|L^{-2}(\<1>_L-\<1>_{L,M})(L^{-1}\cdot )   \|_{H^{-\eta}(\T^2_L)}  \ge \eps    \Big\} \to 0
\end{align*}

\noi
as $M \to \infty$.  It follows  from \eqref{Up1}, \eqref{ch0}, \eqref{UP22}, and \eqref{SL} that  
\begin{align}
\log Z_L(S_{L}^c) & \le  \sup_{W \in \mathcal{H}^1(\T^2_{L^2}) } \E \bigg[-\textbf{V}^L\big( (\<1>_L-\<1>_{L,M} )+W_L   \big)\ind_{  \big\{L^{-2} (\<1>_L-\<1>_{L,M}+W_L)(L^{-1} \cdot) \notin  S_{L}  \big\} } \notag \\
&\hphantom{XXXXXXXXX}-\frac 12 \| \<1>_{L,M} \|_{H^1(\T^2_L)}^2-\frac 12 \| W_L\|_{H^1(\T^2_L)}^2-\int_{\T^2_L}\jb{\nb}\<1>_{L,M} \jb{\nb} W_L dx \bigg] \notag \\
& \le  \sup_{W \in \mathcal{H}^1(\T^2_{L^2}) }  \E \bigg[-\textbf{V}^L\big( (\<1>_L-\<1>_{L,M} )+W_L   \big)\ind_{  \big\{L^{-2} (\<1>_L-\<1>_{L,M}+W_L)(L^{-1} \cdot) \notin  S_{L}  \big\} }  \notag \\
&\hphantom{XXXXXXXXXXXXX} +\Big(c(\zeta)-\frac 12\Big) \| \<1>_{L,M}\|_{H^1(\T^2_L)}^2-\frac {1-\zeta}2 \| W_L\|_{H^1(\T^2_L)}^2 \bigg]
\label{Up2}
\end{align}

\noi
for arbitrary small $\zeta>0$, where in the last line we used the fact that  Young's inequality gives 
\begin{align*}
\bigg| \int_{\T^2_L} \jb{\nb} \<1>_{L,M}  \jb{\nb}W_L dx \bigg| \le c(\zeta) \| \<1>_{L,M} \|_{H^1(\T^2_L)}^2+\frac {\zeta}2 \| W_L\|_{H^1(\T^2_L)}^2.
\end{align*}

\noi
By proceeding as in \eqref{U1} together with \eqref{SL}, \eqref{UP22}, and \eqref{Up2}, we have
\begin{align}
&\log Z_L(S_{L}^c) \notag \\
&\le  \sup_{W\in \mathcal{H}^1(\T^2_{L^2})} \E\Bigg[
\bigg( -\frac{\s}{3}\int_{\T^2_{L}}  W_L^3 dx-\frac{A-\zeta}{2}\|W_L\|_{L^2(\T^2_{L})}^4 \notag \\
&\hphantom{XXXXXXXXX}-\frac {1-\zeta}2 \| W_L\|_{H^1(\T^2_{L} )}^2 \bigg)\ind_{  \big\{L^{-2} (\<1>_L-\<1>_{L,M}+W_L)(L^{-1} \cdot) \notin  S_{L}  \big\} }  \Bigg]+O(L^2M^2)+O(\zeta^{-1}) \notag \\
&\le  L^4 \sup_{W\in \mathcal{H}^1(\T^2_{L^2})} \E\Bigg[
\bigg( -\frac{\s}{3}\int_{\T^2_{L^2}}  W^3 dx-\frac{A-\zeta}{2}\|W\|_{L^2(\T^2_{L^2})}^4 \notag \\
&\hphantom{XXXXXXXXXXX}-\frac {1-\zeta}2 \| W\|_{H^1(\T^2_{L^2} )}^2 \bigg)\ind_{  \big\{ W \in \W_{L}  \big\} } \Bigg]+O(L^2M^2)+O(\zeta^{-1}) \notag \\
&\le - L^4 \inf_{\substack{W\in H^1(\T^2_{L^2}),\\ \|W \|_{H^1(\T^2_{L^2})} \ge \frac \eps 2 } }H_{L^2}^\zeta(W) +O(L^2M^2)+O(\zeta^{-1}).
\label{Up3}
\end{align}

\noi
where
\begin{align*}
H^\zeta_{L^2}(W)=\frac{1-\zeta}{2}\int_{\T^2_{L^2}} |\nb W| ^2 dx+\frac{\s}{3}\int_{\T^2_{L^2}}W^3 dx+\frac{A-\zeta}{2} \int_{\T^2_{L^2}}W^4 dx.
\end{align*}

\noi
Thanks to the GNS inequality on $\T^2_L$ (Lemma \ref{LEM:GNST}) and Young's inequality, we have 
\begin{align*}
H_L(\varphi)\ge \frac {1-\dl}2 \int_{\T^2_L} |\nb \varphi|^2 dx +(A-c(\dl)-c(L)) \bigg( \int_{\T^2_L} \varphi^2 dx \bigg)^2 \ge 0,
\end{align*}

\noi
where we used the fact that $A \ge A_0$ for some sufficiently large $A_0>0$ and $c(L) \to 0$ as $L\to \infty$. This implies that there exists a constant $c$ independent of $L\ge 1$ such that 
\begin{align*}
H^\zeta_{L^2}(W) \ge    c \| \nb W \|_{L^2(\T^2_{L^2})}^2+c\| W\|_{L^2(\T^2_{L^2})}^4, 
\end{align*}

\noi
from which we have
\begin{align}
\inf_{\substack{W\in H^1(\T^2_{L^2}),\\ \|W \|_{H^1(\T^2_{L^2})} \ge \eps/2 } } H^\zeta_{L^2}(W) \ge c \inf_{\substack{W\in H^1(\T^2_{L^2}),\\ \|W \|_{H^1(\T^2_{L^2})} \ge \eps/2 } } \| W\|_{H^1(\T^2_{L^2})}^4 \ge c \eps^4
\label{ASRV1}
\end{align}

\noi
where in the first inequality we used the fact that the infimum is attained when $\|W \|_{H^1(\T^2_{L^2})}$ is equal to $\eps/2$ and so $\| \nb W \|_{L^2(\T^2_{L^2})}^2  \ge \| \nb W \|_{L^2(\T^2_{L^2})}^4  $ if $\eps$ is sufficiently small. It follows from \eqref{Up3}, \eqref{ASRV1}, and taking the limit $L\to \infty$ that 
\begin{align*}
\limsup_{L\to \infty }\frac{\E_{\mu_L} \Big[  \exp\big\{-\textbf{V}_L(\phi) \big \} \ind_{ \{ \phi \not \in S_{L} \} } \Big]}{L^4} \le -c\eps^4.
\end{align*}

\noi
This completes the proof of Lemma \ref{LEM:STAD}.

\end{proof}

\subsection{Lower bound for the free energy}
In this subsection, we derive a lower bound for the free energy.


\begin{lemma}\label{LEM:LB}
There exists a large constant $A_0 \ge 1$ independent of $L \ge 1$ such that for all $\s \in \R \setminus \{0\} $ and $A \ge A_0$, we have
\begin{align*}
\liminf_{L\to \infty} \frac{\log Z_L}{L^4}\ge -\inf_{W\in H^1(\R^2)} H(W).
\end{align*}

\end{lemma}

\begin{proof}
Thanks to Proposition \ref{PROP:Gamma}, we have
\begin{align}
\log Z_L&=\sup_{\dr_L \in \Ha }\E \bigg[-\textbf{V}^L(\<1>_L+\Dr_L)-\frac 12 \int_0^1 \|\dr_L(t) \|^2_{L^2(\T^2_L)} \bigg] 
\label{L1}
\end{align}

\noi
We choose a specific drift $\dr^0_L \in \Ha$, defined by
\begin{align}
\dr^0_L(t)=\frac{1}{\eps} \ind_{ \{ t > 1-\eps \} } \jb{\nb}(-Z_{M,L}+W_L)
\label{L2}
\end{align}

\noi
where 
\begin{align*}
Z_{M,L}:&=\sum_{\substack{\ld \in \Z^2_L\\ |\ld |\le M}} \ft {\<1>}_{\ld}(1-\eps)e^{\ld}_L,\\
W_L:&=L^2W(L\cdot)
\end{align*}

\noi
for any fixed $W\in H^1(\T^2_{L^2})$. Thanks to the time cutoff $\ind_{\{t > 1-\eps\}}$  and the definition of $Z_{M,L}$, the drift $\dr_L$ belongs to the right $\Ha$. Then, by the definition of $\Dr_L(t)$ in \eqref{DR}, we have 
\begin{align}
\Dr^{0}_L(1)=\int_0^1 \jb{\nb}^{-1} \dr_L^0(t) dt=-Z_{M,L}+W_L.
\label{LL3}
\end{align}

\noi 
It follows from \eqref{L1}, \eqref{L2}, and \eqref{LL3} that 
\begin{align}
\log Z_{L} & \ge  \E \bigg[-\textbf{V}^L\big( (\<1>_L-Z_{M,L} )+W_L   \big) -\frac 12 \| Z_{M,L}\|_{H^1(\T^2)}^2-\frac 12 \| W_L\|_{H^1(\T^2)}^2 \notag \\
&\hphantom{XXXXXX}-\int_{\T^2_L}\jb{\nb}Z_{M,L} \jb{\nb} W_L dx \bigg] \notag \\
& \ge  \E \bigg[-\textbf{V}^L\big( (\<1>_L-Z_{M,L} )+W_L   \big)-c(\dl) \| Z_{M,L}\|_{H^1(\T^2)}^2-\frac {1+\dl}2 \| W_L\|_{H^1(\T^2_L)}^2 \bigg]
\label{L3}
\end{align}

\noi
where we used the fact that  Young's inequality gives 
\begin{align*}
\bigg| \int_{\T^2_L} \jb{\nb}Z_{L,M}  \jb{\nb}W_L dx \bigg| \le c(\dl) \| Z_{L,M} \|_{H^1(\T^2_L)}^2+\frac {\dl}2 \| W_L\|_{H^1(\T^2_L)}^2.
\end{align*}

\noi
Notice that $X_{L,M}:=\<1>_L-Z_{M,L}$ is a new Gaussian process and so the Wick powers in $\textbf{V}^L\big((\<1>_L-Z_{M,L}   ) +W_L \big)=\textbf{V}^L\big(X_{L,M} +W_L \big)$ have to be taken with respect to the new Gaussian reference measure. Note that for any Gaussian $X$ and $\s_1, \s_2$ in $\R$, 
\begin{align}
H_1(X;\s_1)&=H_1(X;\s_2) \notag  \\
H_2(X;\s_1)&=H_2(X;\s_2)-(\s_1-\s_2 \notag )\\
H_3(X;\s_1)&=H_3(X;\s_2)-3(\s_1-\s_2)H_1(X,\s_2) \label{LC3}
\end{align}

\noi
where $H_k(x;\s)$ is the Hermite polynomial of degree $k$. 
It then follows from \eqref{LC3} and the Wick powers $:\! X_{L,M}^2 \!:_M, :\! X_{L,M}^2 \!:_M$ that:
\begin{align}
&\int_{\T^2_L}  :\! \big( (\<1>_L-Z_{M,L})  +W_L\big)^3 \!:  dx=\int_{\T^2_L}  :\! ( X_{L,M}+W_L)^3 \!:_M  dx-3C_M\int_{\T^2_L} (X_{L,M}+W_L  ) dx   \notag \\
&=\frac {\s}3\int_{\T^2_L} :\! X_{L,M}^3 \!:_M dx+\s \int_{\T^2_L}  :\! X_{L,M}^2 \!:_M   W_{L} dx+\s \int_{\T^2_L}  X_{L,M} W_{L}^2 dx+\frac{\s}{3}\int_{\T^2_L}  W_{L}^3 dx \notag \\
&\hphantom{XX}-3C_M\int_{\T^2_L} (X_{L,M}+W_L  ) dx,
\label{PAR11}
\end{align}

\noi
and
\begin{align}
\int_{\T^2_L} :\! \big((\<1>_L-Z_{M,L})  +W_L \big)^2 \!: dx&=\int_{\T^2_L} :\! (X_{L,M}+W_L )^2 \!:_M dx-C_M \notag \\
&=\int_{\T^2_{L}} :\! X_{L,M}^2 \!:_Mdx +2 \int_{\T^2_L} X_{L,M} W_L dx+ \int_{\T^2_L} W_{L}^2 dx  -C_M 
\label{PAR22}
\end{align}

\noi
where
\begin{align}
C_M:&=\lim_{N \to \infty }\bigg( \sum_{\substack{n\in \Z^2\\ |n|\le LN}}\frac{1}{\jb{\frac{n}{L}}^2}\frac{1}{L^2}-\Big(\sum_{\substack{n\in \Z^2\\ LM<|n|\le LN}}\frac{1}{\jb{\frac{n}{L}}^2}\frac{1}{L^2} -\eps  \sum_{\substack{n\in \Z^2\\ |n|\le LM}}\frac{1}{\jb{\frac{n}{L}}^2}\frac{1}{L^2}   \Big) \bigg) \notag \\
&=(1+\eps)\sum_{\substack{n\in \Z^2\\ |n|\le LM}}\frac{1}{\jb{\frac{n}{L}}^2}\frac{1}{L^2}\sim\int_{\R^2} \ind_{\{ |y|\le M  \}} \frac{dy}{1+|y|^2} \sim \log M
\label{CMM}
\end{align}

\noi
as $M\to \infty$.
Note that $C_M$ in \eqref{CMM} comes from
\begin{align*}
\E |\<1>_{L,N}(x)|^2&=\sum_{\substack{n\in \Z^2\\ |n|\le LN}}\frac{1}{\jb{\frac{n}{L}}^2}\frac{1}{L^2},\\
\E|  X_{L,N. M}(x) |^2 &= \sum_{\substack{n\in \Z^2\\ LM<|n|\le LN}}\frac{1}{\jb{\frac{n}{L}}^2}\frac{1}{L^2} -\eps  \sum_{\substack{n\in \Z^2\\ |n|\le LM}}\frac{1}{\jb{\frac{n}{L}}^2}\frac{1}{L^2}
\end{align*}

\noi
for any $x\in \T^2_L$, where $X_{L,N,M}:=\<1>_{L,N}-Z_{M,L}$. 
From \eqref{PAR11}, \eqref{PAR22}, Lemma \ref{LEM:Dr1}, and \ref{LEM:Cor00} (i), we have that for arbitrary small $\dl>0$   
\begin{align}
\E\Bigg[ \bigg|\int_{\T^2_L}  :\! \big(X_{L,M}+W_L\big)^3 \!:  dx -\frac{\s}{3}\int_{\T^2_L} W_L^3 dx \bigg| \Bigg] \le \dl  \| W_L\|_{L^2(\T^2_L)}^4+\frac{\dl}{2} \| W_L \|_{H^1(\T^2_L)}^2  +O((\log M)^2L^2)
\label{Up200}
\end{align}

\noi
and
\begin{align}
A\E \Bigg[ \bigg| \int_{\T^2_L} :\! \big(X_{L,M}+W_L \big)^2 \!: dx +C_M\bigg|^2 \Bigg] \le   3A \| W_L \|_{L^2(\T^2_L)}^4 + \dl \| W_L \|_{H^1(\T^2_L)}^2+O(AL^2). 
\label{Up211}
\end{align}

\noi
We also notice that
\begin{align}
\E\Big[\| Z_{L,M} \|_{H^1(\T^2_L)}^2\Big]=(1-\eps)\sum_{\substack{\ld \in \Z^2_L\\ |\ld|\le M}} &=(1-\eps)\sum_{\substack{n\in \Z^2 \\ |n|\le LM}}=O(L^2M^2). 
\label{CABL}
\end{align}

\noi
Hence, it follows from \eqref{L3}, \eqref{Up200}, \eqref{Up211}, \eqref{CABL}, and \eqref{HAn1} that 
\begin{align*}
&\log Z_L\\
&\ge  -\frac{\s}{3}\int_{\T^2_L}  W_L^3 dx-(3A+\dl) \|W_L \|^4_{L^2(\T^2_L)} -\frac{1+\dl}{2}\|W_L \|_{H^1(\T^2_L)}^2  -O(L^2M^2)-O(\dl^{-1})\\
&=-L^4\bigg( \frac{\s}{3}\int_{\T^2_{L^2}}  W^3 dx+(3A+\dl) \|W \|^4_{L^2(\T^2_{L^2})} +\frac{1+\dl}{2}\|W \|_{H^1(\T^2_{L^2})}^2     \bigg)\\
&\hphantom{XXXXXXX}-O(L^2M^2)-O(\dl^{-1})
\end{align*}

\noi
for any $W\in H^1(\T^2_{L^2})$ and $\dl>0$. Hence, we have
\begin{align*}
\log Z_L \ge -L^4\inf_{W \in H^1(\T^2_{L^2})} H_{L^2}^\dl(W)-O(L^2M^2)-O(\dl^{-1}).
\end{align*}

\noi
By taking the limit first in $L\to \infty$ and then $\dl \to 0$ and using Lemma \ref{LEM:GAML}, we obtain
\begin{align*}
\liminf_{L\to\infty}\frac{\log Z_{L}}{L^4} \ge -\inf_{W\in H^1(\R^2)}  H(W),
\end{align*}

\noi
This completes the proof of Lemma \ref{LEM:LB}.  

\end{proof}

Before concluding this subsection, we provide the proofs of the auxiliary lemmas used in the proofs of Lemmas \ref{LEM:UP}, \ref{LEM:STAD}, and \ref{LEM:LB}.


\begin{lemma} \label{LEM:Dr1}
\textup{(i)} Let $\eta>0$. For every $\dl>0$, there exists $c(\dl)>0$ such that
\begin{align}
\bigg| \int_{\T^2_L}  \<2>_{L,N}  \Dr_{L,N}dx  \bigg|
&\le c(\dl) \| \<2>_{L,N} \|_{H^{-\eta}(\T^2_L)}^2  
+ \dl
\| \Dr_{L,N}\|_{ H^{1}(\T^2_L)}^2,  \label{YSS1}  \\
\bigg| \int_{\T^2_L}  \<1>_{L,N} \Dr_{L,N}^2 dx \bigg|
&\le  c(\dl) \| \<1>_{L,N} \|_{H^{-\eta}(\T^2_L)}^4  
+ \dl \Big(
\| \Dr_{L,N}\|_{ H^{1}(\T^2_L)}^2 + 
\| \Dr_{L,N} \|_{L^2(\T^2_L)}^4   \Big).
\label{YSS2}
\end{align}

\noi
for every $N\in \N \cup \{\infty\}$.

\smallskip

\noi
\textup{(ii)}	
Let $A> 0$. Given any small $\eta > 0$, 
there exists $c = c(\eta, A)>0$ such that
\begin{align}
\begin{split}
A\bigg\{ \int_{\T^2_L}&  \Big( \<2>_{L,N} + 2 \<1>_{L,N} \Dr_N + \Dr_N^2 \Big) dx \bigg\}^2 \\
&\ge \frac A4 \| \Dr_{L,N} \|_{L^2(\T^2_L)}^4 - \frac 1{100} \| \Dr_{L,N} \|_{H^1(\T^2_L)}^2 
- c\bigg\{ \| \<1>_{L,N} \|_{H^{-\eta}}^{\frac{4}{1-\eta}}
+ \bigg( \int_{\T^2_L} \<2>_{L,N}  dx \bigg)^2 \bigg\}, 
\end{split}
\label{YY5}
\end{align}

\noi
uniformly in $N \in \N$.

\end{lemma}

\begin{proof}
The proof is a simple application of Young's inequality. For details, see \cite{OSeoT}.
\end{proof}

\section{Collapse of the $\Phi^3_2$-measure}\label{SEC:PRMA}
In this section, we present the proof of Theorem \ref{THM:0}, namely, that the $L$-periodic $\Phi^3_2$-measure exhibits a concentration phenomenon around the minimizer of the Hamiltonian \eqref{H2}. Before proceeding with the proof of Theorem \ref{THM:0}, we first establish the following proposition, which plays a crucial role in the proof of Theorem \ref{THM:0}.

\begin{proposition}\label{PROP:HOM}
There exists a constant $c>0$ independent of $L \ge 1$ such that for any given $\eps>0$ 
\begin{align*}
\rho_L\Big( \big\{ \phi\in H^{-\eta }(\T^2_L):   \| L^{-2}\phi(L^{-1}\cdot) \|_{H^{-\eta}(\T^2_{L^2})} \ge \eps  \big\} \Big) \les \exp\Big\{-c \eps^4 L^{4} \Big\} \to 0
\end{align*}

\noi
as $L\to \infty$.

\end{proposition}

\begin{proof}
We first write 
\begin{align}
\rho_L\Big( \big\{ \phi\in H^{-\eta }(\T^2_L):   \| L^{-2}\phi(L^{-1}\cdot) \|_{H^{-\eta}(\T^2_{L^2})} \ge \eps  \big\} \Big)=\frac{\E_{\mu_L}\Big[  \exp\big\{-\textbf{V}^L(\phi) \big \} \ind_{ \{ \phi \not \in S_{L} \} }  \Big]}{Z_L}
\label{ZZ1}
\end{align}

\noi
where $Z_L$ is the partition function as in \eqref{PART0} and
\begin{align}
S_{L}=\big\{ \phi\in H^1(\T^2_L): \| L^{-2}\phi(L^{-1}\cdot ) \|_{H^{-\eta}(\T^2_{L^2})}   < \eps    \big\}.
\label{ZZ2}
\end{align}

\noi
Hence, from \eqref{ZZ1} and \eqref{ZZ2}, we have
\begin{align}
\log \rho_L(S_{L}^c)&=L^4\Bigg( \frac{\log \E_{\mu_L}\Big[  \exp\big\{-\textbf{V}_L(\phi) \big \} \ind_{ \{ \phi \not \in S_{L} \} }  \Big]}{L^4}-\frac{\log Z_L}{L^4}  \Bigg).
\label{ZZ3}
\end{align}

\noi
It follows from Proposition \eqref{PROP:free} and Lemma \ref{LEM:STAD} that
\begin{align}
\lim_{L\to \infty} \frac{\log Z_L}{L^4}=-\inf_{W\in H^1(\R^2)} H(W)=0
\label{ZZ4}
\end{align}

\noi
and
\begin{align}
\limsup_{L\to \infty}\frac{\log \E_{\mu_L}\Big[  \exp\big\{-\textbf{V}_L(\phi) \big \} \ind_{ \{ \phi \not \in S_{L} \} }  \Big]}{L^4} \le  -c \eps^4.
\label{ZZ5}
\end{align}

\noi
Combining \eqref{ZZ3}, \eqref{ZZ4}, and \eqref{ZZ5}, we obtain
\begin{align}
\rho_L\Big( \big\{ \phi\in H^1(\T^2_L):   \| L^{-2}\phi(L^{-1}\cdot) \|_{H^{-\eta}(\T^2_L)} \ge \eps^4  \big\} \Big)  \les \exp\Big\{- c \eps^4 L^4 \Big\} \to 0
\label{EXP00}
\end{align}

\noi
as $L\to \infty$. This completes the proof of Proposition \ref{PROP:HOM}.

\begin{remark}\rm\label{REM:EXPC}
We can also establish Proposition \ref{PROP:HOM} by restricting to mean-zero fields\footnote{replacing the massive Gaussian free field with a massless Gaussian free field.} and letting  $\eps=L^{-1+\frac{\eta}{2}}$
\begin{align}
\rho_L\Big( \big\{ \phi\in \dot{H}^{-\eta }(\T^2_L):   \| L^{-2}\phi(L^{-1}\cdot) \|_{\dot{H}^{-\eta}(\T^2_{L^2})} \ge L^{-1+\frac{\eta}{2} }  \big\} \Big) \les \exp\Big\{-cL^{2\eta} \Big\} \to 0
\label{HMM1}
\end{align}

\noi
as $L\to \infty$, where $\dot{H}^{-\eta}(\T^2_L)=\big\{\phi \in H^{-\eta}(\T^2_L): \ft \phi(0)=0  \big\}$.  It can be easily shown that 
\begin{align}
\| L^{-2} \phi(L^{-1}\cdot) \|_{\dot{H}^{-\eta}(\T^2_{L^2}) }=L^{-1+\eta} \| \phi \|_{\dot{H}^{-\eta}(\T^2_L) }.
\label{HMM2}
\end{align}

\noi
Hence, by combining \eqref{HMM1} and \eqref{HMM2}, we obtain the exponential concentration of the $L$-periodic $\Phi^3_2$-measure
\begin{align*}
\rho_L\Big( \big\{ \phi\in \dot{H}^{-\eta}(\T^2_L):  \| \phi\|_{\dot{H}^{-\eta}(\T^2_L)} \ge L^{-\frac{\eta}2 }  \big\} \Big) \les \exp\Big\{-cL^{2\eta} \Big\}.
\end{align*}

When considering general fields which are not mean-zero, there is a loss caused by the inhomogeneous component of the $\|\cdot \|_{H^{-\eta}(\T^2_L)}$ norm. In \eqref{HMM2}, the inhomogeneous component only has the factor $L^{-1}$ which is not enough to control $L^{-1+\frac{\eta}{2}}$ in \eqref{HMM1}. Therefore, an additional argument, given below is required to conclude weak convergence to zero.

\end{remark}

\end{proof}

We are now ready to present the proof of Theorem \ref{THM:0}.

\begin{proof}[Proof of Theorem \ref{THM:0}]

We note that
\begin{align*}
\rho_L\Big( \big\{ \phi\in H^{-\eta}(\T^2_L): \max_{1\le j \le m} \big|\jb{\phi,g_j} \big| \ge \eps  \big\} \Big) &\le \sum_{j=1}^m \rho_L\Big( \big\{ \phi\in H^{-\eta}(\T^2_L):  \big|\jb{\phi,g_j} \big| \ge \eps  \big\} \Big) \\
&\le  \frac{m}{\eps} \max_{1\le j \le m} \E_{\rho_L}\Big[  \big| \jb{\phi,g_j} \big|    \Big].
\end{align*}

\noi
In order to estimate $\max_{1\le j \le m} \E_{\rho_L}\Big[  \big| \jb{\phi,g_j} \big|    \Big]$, we first write
\begin{align}
\E_{\rho_L}\Big[  \big| \jb{\phi,g_j} \big|    \Big]&=\E_{\rho_L}\Bigg[ \bigg|  \int_{\T^2_{L^2}}
\jb{\nb}^{-\eta}  (L^{-2}\phi(L^{-1}x) )  \jb{\nb}^{\eta}\big( g_i(L^{-1}x) \big) dx \bigg|  \Bigg] \notag \\
&\le  \int_{|x| \le L}
\E_{\rho_L}\big|\jb{\nb}^{-\eta}  (L^{-2}\phi(L^{-1}x) ) \big|  \big| \jb{\nb}^{\eta}\big( g_i(L^{-1}x) \big) \big| dx  \notag \\
&\hphantom{X}+\int_{L\le |x| \le L^2}
\E_{\rho_L}\big|\jb{\nb}^{-\eta}  (L^{-2}\phi(L^{-1}x) ) \big|  \big| \jb{\nb}^{\eta}\big( g_i(L^{-1}x) \big) \big| dx \notag \\
&= \I+\II. 
\label{PFF6}
\end{align}

\noi
Before we estimate  $\I$ and $\II$ in \eqref{PFF6}, we first consider the expectation $\E_{\rho_L} \Big[ \big\|L^{-2}\phi(L^{-1}\cdot) \big\|^2_{H^{-\eta}(\T^2_{L^2})} \Big]$:
\begin{align}
\E_{\rho_L} \Big[ \big\|L^{-2}\phi(L^{-1}\cdot) \big\|^2_{H^{-\eta}(\T^2_{L^2})} \Big]&=\int_0^\infty \rho_L\Big\{\phi\in H^{-\eta}(\T^2_L): \| L^{-2}\phi(L^{-1}\cdot) \|_{H^{-\eta}(\T^2_{L^2})}^2 >\ld \Big\} d \ld \notag \\
&=\int_0^{L^{-2+\eta}}  \rho_L\Big\{\phi\in H^{-\eta}(\T^2_L): \| L^{-2}\phi(L^{-1}\cdot) \|_{H^{-\eta}(\T^2_{L^2})}^2 >\ld \Big\} d \ld  \notag \\
&\hphantom{X}+\int_{L^{-2+\eta}}^\infty  \rho_L\Big\{\phi\in H^{-\eta}(\T^2_L): \| L^{-2}\phi(L^{-1}\cdot) \|_{H^{-\eta}(\T^2_{L^2})}^2 >\ld \Big\} d \ld \notag \\
&\le L^{-2+\eta} +\int_{L^{-2+\eta}}^\infty  \rho_L\Big\{\phi\in H^{-\eta}(\T^2_L): \| L^{-2}\phi(L^{-1}\cdot) \|_{H^{-\eta}(\T^2_{L^2})}^2 >\ld \Big\} d \ld
\label{PFF1}
\end{align}

\noi
Thanks to Proposition \ref{PROP:HOM}, we have
\begin{align}
&\int_{L^{-2+\eta}}^\infty  \rho_L\Big\{\phi\in H^{-\eta}(\T^2_L): \| L^{-2}\phi(L^{-1}\cdot) \|_{H^{-\eta}(\T^2_{L^2})} >\ld^{\frac 12} \Big\} d \ld \notag \\
&\le  \int_{L^{-2+\eta}}^\infty e^{-\ld^2 L^4} d \ld  \le e^{-\frac{1}{2}L^{2\eta} } \int_{L^{-2+\eta}}^\infty e^{-\frac 12 \ld^2 L^4}  d\ld=  \sqrt{2 \pi}L^2 e^{-\frac{1}{2}L^{2\eta} }.
\label{PFF2}
\end{align}

\noi
Hence, by combining \eqref{PFF1} and \eqref{PFF2}, we have
\begin{align}
\E_{\rho_L} \Big[ \big\|L^{-2}\phi(L^{-1}\cdot) \big\|^2_{H^{-\eta}(\T^2_{L^2})} \Big]\les L^{-2+\eta}.
\label{PFF3}
\end{align}

\noi
Thanks to the spatial stationarity of the measure $\rho_L$, we have that for any fixed $x_1 \in \T^2_L$
\begin{align}
\E_{\rho_L} \Big[ \big\|L^{-2}\phi(L^{-1}\cdot) \big\|^2_{H^{-\eta}(\T^2_{L^2})}  \Big]=L^4\E_{\rho_L}\big|\jb{\nb}^{-\eta}  (L^{-2}\phi(L^{-1}x_1) ) \big|^2.
\label{PFF4}
\end{align}

\noi
It follows from \eqref{PFF3} and \eqref{PFF4} that for any fixed $x_1 \in \T^2_L$ 
\begin{align}
\E_{\rho_L}\big|\jb{\nb}^{-\eta}  (L^{-2}\phi(L^{-1}x_1) ) \big|^2 \les L^{-6+\eta}.
\label{PFF5}
\end{align}

\noi
We now estimate the test function $\big| \jb{\nb}^{\eta}\big( g_i(L^{-1}x) \big) \big| $
on the region $\{ 2^{\ell-1}L\le |x| \le 2^{\ell}L \}$ for $1 \le \ell \le \log L$. Since $g_i$ has compact support in $\T^2_L$, we have that for some large $k \ge 1$
\begin{align}
\big| \jb{\nb}^{\eta}\big( g_i(L^{-1}x) \big) \big| \les 2^{-\ell k}
\label{dec0}
\end{align}

\noi
on the region $\{2^{\ell-1}L\le |x| \le 2^{\ell}L\}$ for $1 \le \ell \le \log L$.

\noi
We are now ready to estimate $\I$ and $\II$ in \eqref{PFF6}.  Let us first consider $\I$. By using the spatial stationarity of the measure $\rho_L$, Cauchy-Schwarz’ inequality, and $L^\infty$ bound of the test function $\jb{\nb}^\eta g_i$, we have that for any $x_1\in \T^2_L$
\begin{align}
\I &\les \Big( \E_{\rho_L}\big|\jb{\nb}^{-\eta}  (L^{-2}\phi(L^{-1}x_1) ) \big|^2  \Big)^{\frac 12}\int_{ |x| \le L} \big| \jb{\nb}^{\eta}\big( g_i(L^{-1}x) \big) \big| dx \notag \\
&\les L^{-3+\frac \eta2 } L^2=L^{-1+\frac{\eta}{2}}.
\label{PFF7}
\end{align}

\noi
It follows from \eqref{dec0}, the spatial stationarity of the measure $\rho_L$, Cauchy-Schwarz’ inequality, and \eqref{PFF5} that for any $x_1\in\T^2_L$ and some large $k\ge 1$
\begin{align}
\II&\le \sum_{\ell=1}^{\log L}\int_{2^{\ell-1}L\le |x| \le 2^{\ell}L}
\E_{\rho_L}\big|\jb{\nb}^{-\eta}  (L^{-2}\phi(L^{-1}x) )   \big| \jb{\nb}^{\eta}\big( g_i(L^{-1}x) \big) \big| dx  \notag \\
&\le \sum_{\ell=1}^{\log L} 2^{-\ell k}\int_{2^{\ell-1}L\le |x| \le 2^{\ell}L}
\E_{\rho_L}\big|\jb{\nb}^{-\eta}  (L^{-2}\phi(L^{-1}x) ) \big|   dx \notag \\
&\le \sum_{\ell=1}^{\log L} 2^{-\ell k} (2^\ell L)^2 \Big( \E_{\rho_L}\big|\jb{\nb}^{-\eta}  (L^{-2}\phi(L^{-1}x_1) ) \big|^2  \Big)^{\frac 12} \notag \\
& \le L^{-3+\frac{\eta}{2}}  \sum_{\ell=1}^{\log L} 2^{-\ell k} (2^\ell L)^2 \notag \\
&\les L^{-1+\frac{\eta}{2}}. 
\label{PFF8}
\end{align}

\noi
By combining \eqref{PFF6}, \eqref{PFF7}, and \eqref{PFF8}, we have
\begin{align*}
\sup_{1\le i \le m}\E_{\rho_L}\Big[  \big| \jb{\phi,g_j} \big|    \Big] \les L^{-1+\frac \eta2}.
\end{align*}

\begin{align}
\rho_L\Big( \big\{ \phi\in H^{-\eta}(\T^2_L): \max_{1\le j \le m} \big|\jb{\phi,g_j} \big| \ge \eps  \big\} \Big) &\le  \frac{m}{\eps} \max_{1\le j \le m} \E_{\rho_L}\Big[  \big| \jb{\phi,g_j} \big|    \Big] \notag \\
&\les \frac{m}{\eps} L^{-1+\frac{\eta}{2}} \to 0 
\label{PFF9}
\end{align}

\noi
as $L\to \infty$. Proposition \ref{PROP:free} and \eqref{PFF9} imply that we complete the proof of Theorem \ref{THM:0}.

\end{proof}

\begin{remark}\rm
By letting $\eps=L^{-1+\eta}$ in \eqref{PFF9}, one derives a quantified version of the concentration. It is not clear whether the resulting rate of concentration is optimal. We do not pursue optimizing the rate of concentration here.


\end{remark}

\begin{ackno}\rm
The work of P.S. is partially supported by NSF grants DMS-1811093 and DMS-2154090. 
\end{ackno}

\end{document}